\documentclass[a4paper,12pt]{article}
\usepackage{amsmath,amssymb,amsthm,graphics,latexsym,amsfonts}
\usepackage{fancyhdr}
\usepackage{color}
\usepackage{graphics}
\usepackage{epsfig}
\usepackage{epstopdf}
\renewcommand{\figurename}

\newtheorem{theorem}{Theorem}[section]
\newtheorem{lemma}[theorem]{Lemma}
\newtheorem{corollary}[theorem]{Corollary}

\newtheorem{conjecture}[theorem]{Conjecture}

\usepackage{indentfirst}
\usepackage{geometry}

\begin{document}
\title{\Large Gallai's path decomposition conjecture for cartesian product of graphs (\uppercase\expandafter{\romannumeral 2})\footnote{Research supported by NSFC (No.12061073)}}

\author{ {Xiaohong Chen\footnote{Email: xhongchen0511@163.com (X. Chen)}, Baoyindureng Wu\footnote{Corresponding author.
Email: baoywu@163.com (B. Wu) }  }\\
\small  \small College of Mathematics and System Science, Xinjiang
University \\ \small  Urumqi 830046, P.R.China \\}
\date{}

\maketitle

{\small \noindent{\bfseries Abstract}:
Let $G$ be a graph of order $n$. A path decomposition $\mathcal{P}$ of $G$ is a collection of edge-disjoint paths that covers all the edges of $G$. Let $p(G)$ denote the minimum number of paths
needed in a path decomposition of $G$. Gallai conjectured that if $G$ is connected, then $p(G)\leq \lceil\frac{n}{2}\rceil$. In this paper, we prove that Gallai's path decomposition conjecture holds for the cartesian product $G\Box H$, where $H$ is any graph and $G$ is a unicyclic graph or a bicyclic graph.

\section {\large Introduction}

All graphs considered here are finite, undirected and simple. We refer to \cite{Bondy} for unexplained terminology and notation. Let $G=(V(G), E(G))$ be a graph. The order $|V(G)|$ and size $|E(G)|$ are denoted by $n(G)$ and $e(G)$, respectively. The degree and the neighborhood of a vertex $v$ are denoted by $d_G(v)$ and $N_G(v)$, respectively. A vertex is called {\it odd} or {\it even} depending on whether its degree is odd or even, respectively. A graph in which every vertex is odd or even is called an {\it odd graph} or an {\it even graph}. The number of odd vertices and even vertices of $G$ are denoted by $n_o(G)$ and $n_e(G)$, respectively. Let $S\subseteq V(G)$. The subgraph induced by $S$, denoted by $G[S]$, is the graph with $S$ as its vertex set, in which two vertices are adjacent if and only if they are adjacent in $G$.  The {\it union} of simple graphs $G$ and $H$ is the graph $G\cup H$ with vertex set $V(G) \cup V(H)$ and edge set $E(G) \cup E(H)$. We use $G+uv$ to denote the graph that arises from $G$ by adding an edge $uv\notin E(G)$, where $u, v \in V(G)$. Similarly, $G-u$ is a graph that arises from $G$ by deleting the vertex $u \in V(G)$. As usual, we use $P_n$ and $C_n$ to denote the path and cycle of order $n$, respectively.

A {\it path decomposition $\mathcal{P}$} of a graph $G$ is a collection of edge-disjoint paths that covers all the edges of $G$. We use $p(G)$ to denote the minimum number of paths needed for a path decomposition of $G$. Erd\H{o}s asked that what is the minimum number of paths into which every connected graph can be decomposed. As a response to the question of Erd\H{o}s, Gallai made the following conjecture:

\begin{conjecture} [\cite{Lovasz1968}]\label{1.1}
For any connected graph $G$ of order $n$, then $p(G)\leq \lceil\frac{n}{2}\rceil$.
\end{conjecture}

The assertion of Gallai's conjecture is sharp if it is true: If $G$ is a graph in which every vertex has odd degree, then in any path decomposition of $G$ each vertex must be the end vertex of some path, and so at least $\lceil\frac{n}{2}\rceil$ paths are required. In 1968, Lov\'{a}sz \cite{Lovasz1968} proved that every graph of order $n$ can be decomposed into at most $\frac{n}{2}$ paths and cycles. Moreover, he proved that if a connected graph $G$ with $n$ vertices has at most one even vertex, then $p(G)\leq \lceil \frac{n}{2}\rceil$.
In 1996, Pyber \cite{Pyber1996} extended Lov\'{a}sz's result by proving that Conjecture 1.1 holds for graphs whose each cycle contains a vertex of odd degree. In 1980, Donald \cite{Donald1980} showed that if $G$ is allowed to be disconnected, then $p(G)\leq \lfloor\frac{3n}{4}\rfloor$, which was improved to $\lfloor\frac{2n}{3}\rfloor$
independently by Dean and Kouider \cite{Dean2000} and Yan \cite{Yan1998}.

Furthermore, the Conjecture 1.1 was verified for several classes of graphs.  Fan \cite{Fan2005} proved that the conjecture is true if each block of the even subgraph is a triangle-free graph with maximum degree at most
three. Botler and Sambinelli \cite{Botler2020(1)} generalized Fan's result in \cite{Fan2005}.
Also, Botler and
Jim\'{e}nez \cite{Botler2017} showed Conjecture 1.1 is ture for a family of even regular graphs with a high girth condition. Harding and McGuinness\cite{Harding2014} proved that for any simple graph $G$ having girth $g\geq 4$, there is a path decomposition of $G$ having at most $\frac{n_o(G)}{2}+\lfloor(\frac{g+1}{2g})n_e(G)\rfloor$ paths.
Recently, Chu, Fan and Zhou \cite{Chu2022} proved that for any triangle-free graph $G$, $p(G)\leq \lfloor\frac{3n}{5}\rfloor$. More results regarding Conjecture 1.1 can be found in \cite{Bonamy2019, Botler2019, Botler2020(2), Chu2021, Fan2005, Geng2015, Jimenez2017}.
But in general, it is still open.

The {\it cartesian product} of simple graphs $G$ and $H$ is the graph $G\Box H$ whose vertex set is $V(G)\times V(H)$ and whose edge set is the set of all pairs $(u_{1} ,v_{1})(u_{2} ,v_{2})$
such that either $u_{1}u_{2}\in E(G)$ and $v_{1}=v_{2}$, or $v_{1}v_{2}\in E(H)$ and $u_{1}=u_{2}$. Every connected graph has a unique factorization under this graph product \cite{Sabidussi1960}, and this factorization can be found in linear time and space \cite{Imrich2007}. In 2016, Jeevadoss and Muthusamy \cite{Jeevadoss2016} shown that the cartesian product of complete graphs can be decompose to $P_5$ and $C_4$. In 2021, Ezhilarasi and Muthusamy \cite{Ezhilarasi2021} proved that cartesian product of complete graphs can be decompose to $P_5$ and $K_{1,4}$. In 2018, Oyewumi, Akwu and Azer \cite{Oyewumi2018} determined a path decomposition of cartesian product of $P_m$ and $C_n$ with $p(P_m\Box C_n)=n$. In 2023, Chen and Wu \cite{Chen2023} extended the above result further and proved the following theorems.

\begin{theorem} [\cite{Chen2023}] \label{1.2}
Let $m$ be an integer $m\geq 2$ and $H$ be a connected graph of order $n$. If $G=P_{m}\Box H$, then
$p(G)\leq \frac{mn}{2}$.
\end{theorem}

\begin{theorem} [\cite{Chen2023}] \label{1.3}
Let $G$ be a connected graph of order $m\geq 2$ and $H$ be a connected graph of order $n$. If $p(G)=\frac{n_o(G)}{2}$, then $p(G\Box H)\leq \frac{mn}{2}$.
\end{theorem}

In addition, for other interesting results on cartesian product of graphs can be found in \cite{Anderson2022, Hammack2011, Imrich2008, Kaul2023, Rall2023}.

A {\it unicyclic graph} is a connected graph with $n$ vertices and $n$ edges. A {\it bicyclic graph} is a connected graph with $n$ vertices and $n+1$ edges. We refer to \cite{Azari2023, Barik2023, Eliasi2022, Li2023, Liu2023} for some recent results on unicyclic and bicyclic graphs. In this paper, we prove that Gallai's path decomposition conjecture holds for the cartesian product of a general graph and a unicyclic graph, a general graph and a bicyclic graph, respectively.

%If $m=n+c-1$, then $G$ is called a c-cyclic graph. Specially, if $c=1$ and $2$, then $G$ is called a {\it unicyclic graph} and a {\it bicyclic graph}, respectively.

%In this paper, we prove that Gallai's path decomposition conjecture holds for the Cartesian product of a general graph and a cycle
%The paper is organized as follows. In Section 2, we present some definitions,
%notation and state an auxiliary result needed in the proof of our main results. In Section 3, we prove that Conjecture 1.1 holds for the Cartesian product of a general graph and a cycle.

\section{\large Notation}

In this section, we will introduce some notations to be used in this paper. Let $u,v\in V(G)$ and $P_{u,v}$ be a path of $G$ with ends $u$ and $v$. We call $u$ is an {\it odd end vertex} of $P_{u,v}$ if $u$ is odd in $G$, otherwise, $u$ is an {\it even end vertex} of $P_{u,v}$. If both $u$ and $v$ are odd end vertices in $G$, we call $P_{u,v}$ is an {\it odd-odd path}. If both $u$ and $v$ are even end vertices in $G$, then $P_{u,v}$ is an {\it even-even path}. If the degree of $u$ and $v$ have different parity, we call $P_{u,v}$ is an {\it odd-even path}.

A sequence of vertices and edges $W := v_1 e_1 v_2 ...v_{l} e_l v_{l+1}$ is called a {\it trail} if $e_i=v_{i}v_{i+1}\in E(W)$ for each $i\in \{1,2,\cdots,l\}$ and $e_i\neq e_j$ if $i\neq j$. If $v_1=v_{l+1}$, then $W$ is a {\it closed trail}. A trail $W := v_1 P_{v_1,v_2} v_2 ...v_{r} P_{v_r,v_{r+1}} v_{r+1}$ is called a {\it$\mathcal{P}$-$r$-trail} if $W$ is the union of $r$ edge disjoint paths $P_{v_1,v_2} ... P_{v_r,v_{r+1}}$ of $\mathcal{P}$, and $v_i\neq v_j$ if $i\neq j$. If $v_1=v_{r+1}$, then $W$ is {\it $\mathcal{P}$-$r$-closed-trail}.

Let $P_m$ be a path with $V(P_m)=\{u_1,u_2,\cdots,u_m\}$. Some internal vertices of $P_m$ are called {\it virtual}, and the remaining ones are called {\it real}. Note that the ends of the path must be real. Such a path is called a {\it virtual-real} path. Similarly, $C_m$ be a cycle with $V(C_m)=\{u_1,u_2,\cdots,u_m\}$. Some vertices of $C_m$ are called {\it virtual}, and the remaining ones are called {\it real}. Note that the number of real vertices on $C_m$ is at least three. Such a cycle is called a {\it virtual-real} cycle. The cartesian product of a virtual-real path $P_m$ (or virtual-real cycle $C_m$) and a graph $H$ is defined as the graph with $V(P_m\boxdot H)=V(P_m)\times V(H)$ (or $V(C_m\boxdot H)=V(C_m)\times V(H)$), in which for any $v\in V(H)$, $\{(u_i, v):\ 1\leq i\leq m\}$ induces $P_m$ (or $C_m$), for each $i\in\{1, \ldots, m\}$, $\{u_i\}\times V(H)$ induces a graph isomorphic to $H$ if $u_i$ is real, and otherwise, $\{u_i\}\times V(H)$ is an independent set.

\vspace{2mm} The following lemma presented a path decomposition for a graph and was proved in \cite{Chen2023}. We restate it as:

\begin{lemma} [\cite{Chen2023}] \label{2.1} For any connected graph $H$, there exists a path decomposition $\mathcal{P}(H)$ of the edges of $H$, such that:

(1) each odd vertex is the end vertex of exactly one path of $\mathcal{P}(H)$;

(2) each even vertex is the end vertex of exactly two path of $\mathcal{P}(H)$;

(3) $|\mathcal{P}(H)|=\frac{n_o(H)}{2}+n_e(H)$.
\end{lemma}

Let $\mathcal{P}(H)$ be a path decomposition of $H$ as given in the Lemma \ref{2.1}. Now the $\mathcal{P}(H)$ can be classified into odd-odd, even-even, odd-even (or even-odd) paths respectively.
We denote the number of odd end vertices of odd-odd paths in $\mathcal{P}(H)$ as $n_{o}^{1}(H)$, and the number of odd end vertices of odd-even paths as $n_{o}^{2}(H)$. Clearly, $n_{o}^{1}(H)+n_{o}^{2}(H)=n_o(H)$.
For any even-even path, it is either on a $\mathcal{P}$-$r$-trail or on a $\mathcal{P}$-$r$-closed-trail. We denote the number of even end vertices of even-even paths on the $\mathcal{P}$-$r$-trail in $\mathcal{P}(H)$ as $n_{e}^{1}(H)$, and the number of even end vertices of even-even paths on the $\mathcal{P}$-$r$-closed-trail in $\mathcal{P}(H)$ as $n_{e}^{2}(H)$. Clearly, $n_{e}^{1}(H)+n_{e}^{2}(H)=n_e(H)$.

Based on the above facts, it follows that $H$ can be decomposed into some odd-odd paths, $\mathcal{P}$-$r$-trail and $\mathcal{P}$-$r$-closed-trail. Let $v$ be any vertex in $H$. It is assigned to be real on the path in $\mathcal{P}(H)$ with $v$ as an end vertex, and is assigned to virtual for the remaining paths containing it. Then $2t_1+\sum_{i=1}^{t_2}(r_i+1)+\sum_{j=1}^{t_3}r_j=n(H)$, where $t_1=\frac{n_{o}^{1}(H)}{2}$, $t_2=\frac{n_{o}^{2}(H)}{2}$, and $t_3$ is the number of $\mathcal{P}$-$r$-closed-trails of $\mathcal{P}(H)$. Two examples are shown in Fig. 1. Let $H_1$ be the union of odd-odd paths of $\mathcal{P}(H)$, $H_2$ be the union of $\mathcal{P}$-$r$-trails of $\mathcal{P}(H)$ and $H_3$ be the union of $\mathcal{P}$-$r$-closed-trail of $\mathcal{P}(H)$. It obvious that $E(H)=\bigcup_{i=1}^{3}E(H_i)$. Hence, $E(G\Box H)=\bigcup_{i=1}^{3}E(G\boxdot H_i)$.

\begin{figure}[h]
\begin{center}
\includegraphics[width=13cm]{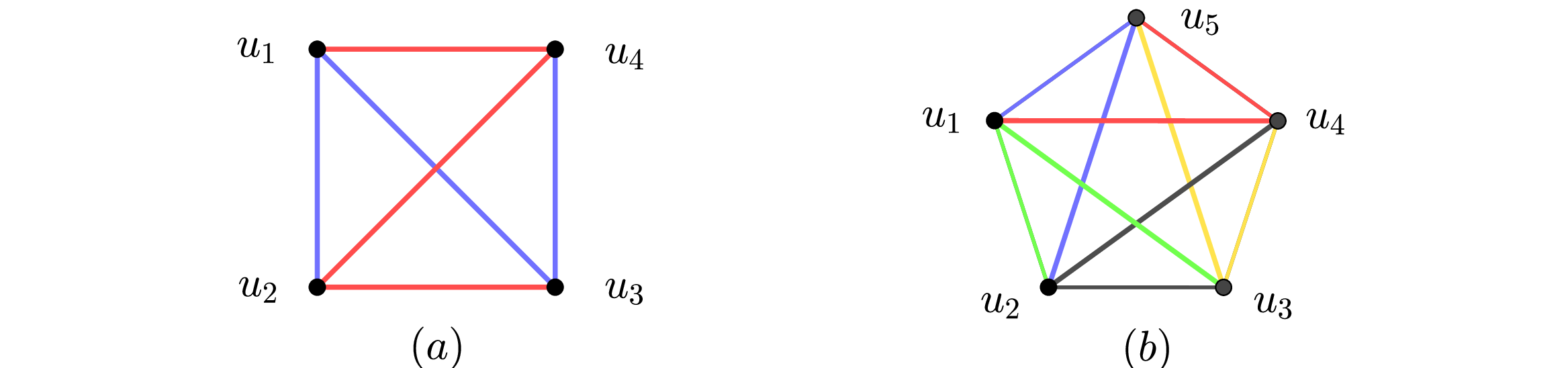}\\
\end{center}
\vspace{2mm}\footnotesize{Fig. 1. $(a)$ represents a path decomposition of $K_4$ that satisfies Lemma 2.1, where $P_{u_1,u_3}=u_1u_4u_2u_3$ that $u_1$ and $u_3$ are real and $u_2$ and $u_4$ are virtual, and $P_{u_2,u_4}=u_2u_1u_3u_4$ that $u_2$ and $u_4$ are real and $u_1$ and $u_3$ is virtual. $(b)$ represents a path decomposition of $K_5$ that satisfies Lemma 2.1, for each $i\in \{1,2,3,4,5\}$, $P_{u_i,u_{i+1}}=u_iu_{i-1}u_{i+1}$ where $u_i$ and $u_{i+1}$ are real and $u_{i-1}$ is virtual and the subscripts are taken modular $5$.}
\end{figure}

In this paper we also use the following lemma.

\begin{lemma} [\cite{Chen2023}] \label{1}
If $G'$ be a graph obtained from a graph $G$ by subdividing an edge, then $p(G')=p(G)$.
\end{lemma}

\section{\large Some auxiliary results}

\vspace{2mm} In this section we will use the path decomposition of cartesian product of $P_m$ and a connected graph $H$ given in \cite{Chen2023}, which is as follows.

Let $P_m:=u_1u_2\cdots u_m$ and $V(H)=\{v_1,v_2,\cdots,v_{n}\}$. For the graph $G=P_m\Box H$, let $U_i=\{(u_i,v_j): j\in \{ 1,2,\cdots,n\}\}$ where $i\in\{1,2,\cdots,m\}$. Let $\mathcal{P}(H)$ be a path decomposition of $H$ as described in Lemma \ref{2.1} and $\mathcal{P}_i=\{P^i: V(P^i)=\{(u_i,v): v\in V(P), P\in \mathcal{P}(H)\}\}$, where $i\in\{1,2,\cdots,m\}$.

\vspace{2mm} For an odd-odd path $P_{y_1,y_2}$ of $\mathcal{P}(H)$, $P_{m}\boxdot P_{y_1,y_2}$ can be decomposed into $m$ paths as follows:

\vspace{2mm} If $m$ is even, then

\vspace{1.5mm} $\mathcal{P}_{11}=\{\bigcup_{i=1}^{m} P_{y_1,y_2}^i\} \cup \{(u_i,y_1)(u_{i+1},y_1): i \in \{2,4, \cdots, m-2 \} \} \cup \{(u_i,y_2)(u_{i+1},y_2): i \in \{ 1,3, \cdots, m-1\}\}$,

$\mathcal{P}_{12}=\{(u_i,y_1)(u_{i+1},y_1)\}$ for each $i \in \{1,3, \cdots, m-1\}$,

$\mathcal{P}_{13}=\{(u_i,y_2)(u_{i+1},y_2)\}$ for each $i \in \{2,4, \cdots, m-2\}$.

\vspace{2mm} If $m$ is odd, then

\vspace{1.5mm} $\mathcal{P}_{11}=\{\bigcup_{i=1}^{m} P_{y_1,y_2}^i\} \cup \{(u_i,y_1)(u_{i+1},y_1): i \in \{2,4, \cdots, m-1 \} \} \cup \{(u_i,y_2)(u_{i+1},y_2): i\in \{1,3, \cdots, m-2\}\}$,

$\mathcal{P}_{12}=\{(u_i,y_1)(u_{i+1},y_1)\}$ for each $i \in \{1,3, \cdots, m-2\}$,

$\mathcal{P}_{13}=\{(u_i,y_2)(u_{i+1},y_2)\}$ for each $ i \in \{2,4, \cdots, m-1\}$.

\vspace{2mm} For a $\mathcal{P}$-$r$-trail $W_1=y_1P_{y_1,y_2}\cdots y_{r}P_{y_{r}y_{r+1}}y_{r+1}$ of $\mathcal{P}(H)$, $P_{m}\boxdot W_1$ can be decomposed into $m+r-1$ paths as follows:

\vspace{2mm} If $m$ is even

\vspace{1.5mm} $\mathcal{P}_{21}=\{\bigcup_{i=1}^{m} P_{y_j,y_{j+1}}^i\} \cup \{(u_i,y_j)(u_{i+1},y_j): i\in \{2,4, \cdots, m-2\}\} \cup \{(u_i,y_{j+1})(u_{i+1},y_{j+1}): i\in \{1,3, \cdots, m-1\}\}$ for each $ j\in\{1, \ldots, r\}$,

$\mathcal{P}_{22}=\{(u_i,y_1)(u_{i+1},y_1)\}$ for each $ i \in \{1, 3, \cdots, m-1\}$,

$\mathcal{P}_{23}=\{(u_i,y_{r+1})(u_{i+1},y_{r+1})\}$ for each $ i \in \{2, 4, \cdots, m-2\}$.

\vspace{2mm} If $m$ is odd

\vspace{1.5mm} $\mathcal{P}_{21}=\{\bigcup_{i=1}^{m} P_{y_j,y_{j+1}}^i\} \cup \{(u_i,y_j)(u_{i+1},y_j): i\in \{2,4, \cdots, m-1\}\} \cup \{(u_i,y_{j+1})(u_{i+1},y_{j+1}): i\in \{1,3, \cdots, m-2\}\}$ for each $j \in \{1, \ldots, r \}$,

$\mathcal{P}_{22}=\{(u_i,y_1)(u_{i+1},y_1)\}$ for each $ i \in \{1, 3, \cdots, m-2\}$,

$\mathcal{P}_{23}=\{(u_i,y_{r+1})(u_{i+1},y_{r+1})\}$ for each $ i \in \{2, 4, \cdots, m-1\}$.

\vspace{2mm} For a $\mathcal{P}$-$r$-closed-trail $W_2=y_1P_{y_1,y_2}\cdots y_{r}P_{y_{r}y_{1}}y_{1}$ of $\mathcal{P}(H)$, $P_{m}\boxdot W_2$ can be decomposed into $r$ paths as follows:

\vspace{2mm}If $m$ is even, then

\vspace{1.5mm} $\mathcal{P}_{31}=\{\bigcup_{i=1}^{m} P_{y_j,y_{j+1}}^i\} \cup \{(u_i,y_j)(u_{i+1},y_j): i\in \{2,4, \cdots, m-2\}\} \cup \{(u_i,y_{j+1})(u_{i+1},y_{j+1}): i\in \{1,3, \cdots, m-1\}\}$ for each $j\in\{1, \ldots, r\}$.

\vspace{2mm}If $m$ is odd, then

\vspace{1.5mm} $\mathcal{P}_{31}=\{\bigcup_{i=1}^{m} P_{y_j,y_{j+1}}^i\} \cup \{(u_i,y_j)(u_{i+1},y_j): i\in \{2,4, \cdots, m-1 \}\} \cup \{(u_i,y_{j+1})(u_{i+1},y_{j+1}): i\in \{1,3, \cdots, m-2\}\}$ for each $j\in\{1, \cdots, r\}$.

\vspace{2mm}To conclude the above, $G$ has a path decomposition $\mathcal{P}(G)$ with
\begin{eqnarray*}
p(G)&\leq&p(P_m\boxdot H_1)+p(P_m\boxdot H_2)+p(P_m\boxdot H_3)\\
&\leq&m\frac{n_{o}^{1}(H)}{2}+m\frac{n_{o}^{2}(H)}{2}+n_e^{1}(H)+n_e^{2}(H)\\
&=& m\frac{n_o(H)}{2}+n_e(H).
\end{eqnarray*}

\begin{center}
\includegraphics[width=13cm]{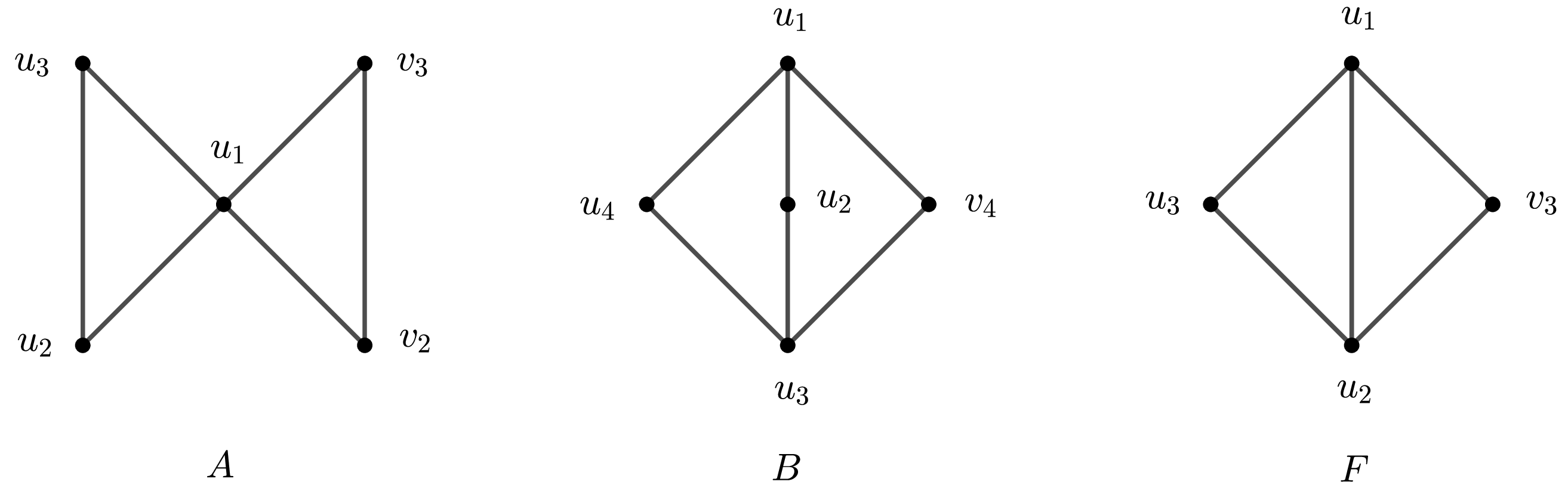}\\
\centerline{Fig. 2 }
\end{center}

\vspace{3mm} Let $\mathcal{A}$ be the family of all graphs, which are obtained from $A$ (in Fig. 2) by identifying a vertex $x\in \{u_2, u_3, v_2, v_3\}$ with an end vertex of a separate path $P$. Let $\mathcal{B}$ be the family of all graphs, which are obtained from $B$ (in Fig. 2) by identifying a vertex $x\in \{u_4, v_4\}$ with an end vertex of a separate path $P$. Let $\mathcal{F}$ be the family of all graphs, which are obtained from $F$ (in Fig. 2) by identifying a vertex $x\in \{u_3, v_3\}$ with an end vertex of a separate path $P$. Let $\mathcal{G}$ be the family of all graphs, which are obtained from a path $P$ by identifying each end vertices of $P$ with a vertex of a $C_3$. By a similar argument as in above, we can obtain the following lemmas.

\begin{lemma}\label{3.1}
Let $G$ be a connected graph of order $m$ and $H$ be a connected graph of order $n$. If $G \in \mathcal{A}$, then $p(G\Box H)\leq \frac{mn}{2}$.
\end{lemma}

\begin{proof}

Let $G=P_m+ \{u_1u_3, u_3u_5\}$ where $P_m=u_1u_2\cdots u_m$, $m\geq 5$.

\vspace{2mm}(1) For an odd-odd path $P_{y_1,y_2}$ of $\mathcal{P}(H)$, $G\boxdot P_{y_1,y_2}$ can be decomposed into $m$ paths as follows:

\vspace{2mm} $\mathcal{P}(G\boxdot P_{y_1,y_2})=\mathcal{P}_{11} \cup \{\mathcal{P}_{12}\setminus\{(u_1,y_1)(u_2,y_1), (u_3,y_1)(u_4,y_1)\}\} \cup \{(u_2,y_1)(u_1,y_1)(u_3,y_1), \\ (u_4,y_1)(u_3,y_1)(u_5,y_1)\} \cup \{\mathcal{P}_{13}\setminus \{(u_2,y_2)(u_3,y_2),(u_4,y_2)(u_5,y_2)\}\} \cup \{(u_1,y_2)(u_3,y_2)(u_2,y_2), \\ (u_4,y_2)(u_5,y_2)(u_3,y_2)\}$.

\vspace{1.5mm} Then $p(G\boxdot H_1)\leq \sum_{i=1}^{t_1} p(G \boxdot P_{y_{1_i},y_{2_i}})= \sum_{i=1}^{t_1} p(G \Box P_2)\leq \frac{mn_{o}^{1}(H)}{2}$.

\vspace{2mm}(2) For a $\mathcal{P}$-$r$-trail $W_1=y_1P_{y_1,y_2}\cdots y_{r}P_{y_{r}y_{r+1}}y_{r+1}$ of $\mathcal{P}(H)$, $G\boxdot W_1$ can be decomposed into $m+2r-2$ paths as follows:

\vspace{2mm} $\mathcal{P}_{24}=\{(u_1,y_j)(u_3,y_j)(u_5,y_j)\}$ for each $ j \in \{2,3, \cdots, r\}$.

$\mathcal{P}(G\boxdot W_1)=\mathcal{P}_{21} \cup \{\mathcal{P}_{22}\setminus\{(u_1,y_1)(u_2,y_1),(u_3,y_1)(u_4,y_1)\}\}
\cup \{(u_2,y_1)(u_1,y_1)(u_3,y_1), \\ (u_4,y_1)(u_3,y_1)(u_5,y_1)\} \cup \{\mathcal{P}_{13}\setminus
\{(u_2,y_{r+1})(u_3,y_{r+1}),(u_4,y_{r+1})(u_5,y_{r+1})\}\} \cup \{(u_1,y_{r+1})\\ (u_3,y_{r+1})(u_2,y_{r+1}),(u_4,y_{r+1})(u_5,y_{r+1})(u_3,y_{r+1})\} \cup \mathcal{P}_{24}$.

\vspace{1.5mm} Then $p(G\boxdot H_2)\leq \sum_{i=1}^{t_2} p(G \boxdot W_{1_i}) \leq \sum_{i=1}^{t_2} (m+2r_i-2)\leq \sum_{i=1}^{t_2}\frac{m(r_i+1)}{2}\leq \frac{m(n_{o}^{2}(H)+n_{e}^{1}(H))}{2}$ where $r_i$ is the number of paths on $W_{1_i}$.

\vspace{2mm}(3) For a $\mathcal{P}$-$r$-closed-trail $W_2=y_1P_{y_1,y_2}\cdots y_{r}P_{y_{r}y_{1}}y_{1}$ of $\mathcal{P}(H)$, $G\boxdot W_2$ can be decomposed into $2r$ paths as follows:

\vspace{2mm} $\mathcal{P}_{32}=\{(u_1,y_j)(u_3,y_j)(u_5,y_j)\}$ for each $ j \in \{1,2, \cdots, r\}$.

$\mathcal{P}(G\boxdot W_2)=\mathcal{P}_{31} \cup \mathcal{P}_{32}$.

\vspace{1.5mm} Then $p(G\boxdot H_3)\leq \sum_{i=1}^{t_3} p(G \boxdot W_{2_i})\leq \sum_{i=1}^{t_3} 2r_i \leq \frac{mn_{e}^{2}(H)}{2}$ where $r_i$ is the number of paths on $W_{2_i}$.

\vspace{2mm}Combining the above, we conclude that there is a path decomposition $\mathcal{P}(G\Box H)$ of graph $G\Box H$ with

\begin{eqnarray*}
p(G\Box H)&\leq&p(G\boxdot H_1)+p(G\boxdot H_2)+p(G\boxdot H_3)\\
&\leq&\frac{mn_{o}^{1}(H)}{2}+\frac{m(n_{o}^{2}(H)+n_{e}^{1}(H))}{2}+\frac{mn_{e}^{2}(H)}{2}\\
&=& \frac{mn}{2}.
\end{eqnarray*}
\end{proof}

\begin{lemma}\label{3.2}
Let $G$ be a connected graph of order $m$ and $H$ be a connected graph of order $n$. If $G \in \mathcal{B}$, then $p(G\Box H)\leq \frac{mn}{2}$.
\end{lemma}

\begin{proof}

Let $G=P_m+ \{u_1u_4, u_2u_5\}$ where $P_m=u_1u_2\cdots u_m$, $m\geq 5$.

\vspace{2mm} (1) For an odd-odd path $P_{y_1,y_2}$ of $\mathcal{P}(H)$, $G\boxdot P_{y_1,y_2}$ can be decomposed into $m-2$ paths as follows:

\vspace{2mm} $\mathcal{P}(G\boxdot P_{y_1,y_2})=\mathcal{P}_{11} \cup \{\mathcal{P}_{12} \setminus \{(u_1,y_1)(u_2,y_1),(u_3,y_1)(u_4,y_1)\}\} \cup \{(u_3,y_1)(u_4,y_1)(u_1,y_1)\\ (u_2,y_1)(u_5,y_1)\} \cup \{\mathcal{P}_{13} \setminus \{(u_2,y_2)(u_3,y_2),(u_4,y_2)(u_5,y_2)\}\} \cup
\{(u_1,y_2)(u_4,y_2)(u_5,y_2)(u_2,y_2)(u_3,y_2)\}$.

\vspace{1.5mm} Then $p(G\boxdot H_1)=\sum_{i=1}^{t_1} p(G \boxdot P_{y_{1_i},y_{2_i}})=\frac{n_{o}^{1}(H)}{2} p(G \Box P_2)\leq \frac{mn_{o}^{1}(H)}{2}$.

\vspace{2mm} (2) For a $\mathcal{P}$-$r$-trail $W_1=y_1P_{y_1,y_2}\cdots y_{r}P_{y_{r}y_{r+1}}y_{r+1}$ of $\mathcal{P}(H)$, we will consider the two cases as follows:

\vspace{2mm}\noindent{\bf Case 1.} If $m\geq 6$, then $G\boxdot W_1$ can be decomposed into $m+3r-5$ paths as follows:

\vspace{1.5mm} $\mathcal{P}_{24}=\{(u_1,y_j)(u_4,y_j),(u_2,y_j)(u_5,y_j)\}$ for each $ j \in \{2,3, \cdots, r\}$.

$\mathcal{P}(G\boxdot W_1)=\mathcal{P}_{21} \cup \{\mathcal{P}_{22} \setminus \{(u_1,y_1)(u_2,y_1),(u_3,y_1)(u_4,y_1)\}\} \cup \{(u_3,y_1)(u_4,y_1)(u_1,y_1)$ $(u_2,y_1)(u_5,y_1)\} \cup \{\mathcal{P}_{23} \setminus \{(u_2,y_{r+1})(u_3,y_{r+1}),(u_4,y_{r+1})(u_5,y_{r+1})\}\} \cup
\{(u_1,y_{r+1})(u_4,y_{r+1})(u_5,y_{r+1})$ $(u_2,y_{r+1})(u_3,y_{r+1})\} \cup \mathcal{P}_{24}$.

\vspace{1.5mm} Then $p(G\boxdot H_2)\leq \sum_{i=1}^{t_2} p(G \boxdot W_{1_i})\leq \sum_{i=1}^{t_2} (m+3r_i-5)\leq \sum_{i=1}^{t_2} \frac{m(r_i+1)}{2}\leq \frac{m(n_{o}^{2}(H)+n_{e}^{1}(H))}{2}$ where $r_i$ is the number of paths on $W_{1_i}$.

\vspace{2mm}\noindent{\bf Case 2.} If $m= 5$, then $P_{m}\boxdot W_1$ can be decomposed into at most $2r+2$ paths as follows:

\vspace{2mm} When $r$ is even,
$\mathcal{P}_{24}=\{(u_1,y_{r})(u_4,y_{r}),(u_2,y_{r})(u_5,y_{r})\} \cup \{(u_2,y_j)(u_5,y_j)P^5_{y_j,y_{j+1}}(u_5,y_{j+1})\\ (u_2,y_{j+1})\} \cup \{(u_4,y_j)(u_1,y_j)P^1_{y_j,y_{j+1}}(u_1,y_{j+1})(u_4,y_{j+1})\}$ for each $j\in \{2,4,\cdots, r-2\}$.

\vspace{2mm} When $r$ is odd,
$\mathcal{P}_{24}=\{(u_2,y_j)(u_5,y_j)P^5_{y_j,y_{j+1}}(u_5,y_{j+1})(u_2,y_{j+1})\}\cup \{(u_4,y_j)(u_1,y_j)\\ P^1_{y_j,y_{j+1}}(u_1,y_{j+1})(u_4,y_{j+1})\}$ for each $j\in \{2,4,\cdots, r-1\}$.

\vspace{2mm} Let $\mathcal{P'}_{21}=\{P': P'=P-\{P\cap Q\}, P\in \mathcal{P}_{21}\ and\ Q\in \mathcal{P}_{24}\}$.

\vspace{2mm} Hence,
$\mathcal{P}(G\boxdot W_1)=\mathcal{P'}_{21} \cup \mathcal{P}_{24} \cup \{(u_3,y_1)(u_4,y_1)(u_1,y_1)(u_2,y_1)(u_5,y_1)\} \cup \{(u_1,y_{r+1})\\ (u_4,y_{r+1})(u_5,y_{r+1})(u_2,y_{r+1})(u_3,y_{r+1})\}$.

\vspace{2mm} Then $p(G\boxdot H_2)\leq \sum_{i=1}^{t_2} p(G \boxdot W_{1_i})\leq \sum_{i=1}^{t_2} (2r_i+2)\leq \sum_{i=1}^{t_2} \frac{5(r_i+1))}{2} = \frac{5}{2}(n_{o}^{2}(H)+n_{e}^{1}(H))$ where $r_i$ is the number of paths on $W_{1_i}$.

\vspace{2mm}(3) For a $\mathcal{P}$-$r$-closed-trail $W_2=y_1P_{y_1,y_2}\cdots y_{r}P_{y_{r}y_{1}}y_{1}$ of $\mathcal{P}(H)$, we will consider the two cases as follows:

\vspace{2mm}\noindent{\bf Case 1.} If $m\geq 6$, then $P_{m}\boxdot W_2$ can be decomposed into $3r$ paths as follows:

\vspace{2mm}$\mathcal{P}_{32}=\{(u_1,y_j)(u_4,y_j),(u_2,y_j)(u_5,y_j)\}$ for each $ j \in \{1,3, \cdots, r\}$.

$\mathcal{P}(G\boxdot W_2)=\mathcal{P}_{31} \cup \mathcal{P}_{32}$.

\vspace{1.5mm} Then $p(G\boxdot H_3)\leq \sum_{i=1}^{t_3} p(G \boxdot W_{2_i}) \leq \sum_{i=1}^{t_3} 3r_i \leq m\frac{n_{e}^{2}(H)}{2}$ where $r_i$ is the number of paths on $W_{2_i}$.

\vspace{2mm}\noindent{\bf Case 2.} If $m= 5$, then $P_{m}\boxdot W_2$ can be decomposed into at most $2r+1$ paths as follows:

\vspace{2mm} When $r$ is odd, $\mathcal{P}_{32}=\{(u_1,y_{r})(u_4,y_{r}),(u_2,y_{r})(u_5,y_{r})\} \cup \{(u_2,y_j)(u_5,y_j)P^5_{y_j,y_{j+1}}(u_5,y_{j+1})\\ (u_2,y_{j+1})\}\cup \{(u_4,y_j)(u_1,y_j)P^1_{y_j,y_{j+1}}(u_1,y_{j+1})(u_4,y_{j+1})\}$ for each $j\in \{1,3,\cdots, r-2\}$.

\vspace{2mm} When $r$ is even, $\mathcal{P}_{32}=\{(u_2,y_j)(u_5,y_j)P^5_{y_j,y_{j+1}}(u_5,y_{j+1})(u_2,y_{j+1})\}\cup \{(u_4,y_j)(u_1,y_j)\\ P^1_{y_j,y_{j+1}}(u_1,y_{j+1})(u_4,y_{j+1})\}$ for each $j\in \{1,3,\cdots, r-1\}$.

\vspace{2mm} Let $\mathcal{P'}_{31}=\{P': P'=P-\{P\cap Q\}, P\in \mathcal{P}_{31}\ and\ Q\in \mathcal{P}_{32}\}$.

$\mathcal{P}(G\boxdot W_2)=\mathcal{P'}_{31} \cup \mathcal{P}_{32}$.

\vspace{1.5mm} Then $p(G\boxdot H_3)\leq \sum_{i=1}^{t_3} p(G \boxdot W_{2_i})\leq \sum_{i=1}^{t_3} (2r_i+1) \leq \frac{mn_{e}^{2}(H)}{2}$ where $r_i$ is the number of paths on $W_{2_i}$.

\vspace{2mm} Summing up the above, $G\Box H$ has a path decomposition $\mathcal{P}(G\Box H)$ with
\begin{eqnarray*}
p(G\Box H)&\leq&p(G\boxdot H_1)+p(G\boxdot H_2)+p(G\boxdot H_3)\\
&\leq&\frac{mn_{o}^{1}(H)}{2}+\frac{m(n_{o}^{2}(H)+n_e^{1}(H))}{2}+\frac{mn_e^{2}(H)}{2}\\
&=& \frac{mn}{2}.
\end{eqnarray*}
\end{proof}

\begin{lemma}\label{3.3}
Let $G$ be a connected graph of order $m$ and $H$ be a connected graph of order $n$. If $G \in \mathcal{F}$, then $p(G\Box H)\leq \frac{mn}{2}$.
\end{lemma}

\begin{proof}

Let $G=P_m+ \{u_1u_3, u_2u_4\}$ where $P_m=u_1u_2\cdots u_m$ and $m\geq 4$.

\vspace{2mm}(1) For an odd-odd path $P_{y_1,y_2}$ of $\mathcal{P}(H)$, $G\boxdot P_{y_1,y_2}$ can be decomposed into $m$ paths as follows:

\vspace{2mm}
$\mathcal{P}(G\boxdot P_{y_1,y_2})=\mathcal{P}_{11} \cup \{\mathcal{P}_{12} \setminus \{(u_1,y_1)(u_2,y_1),(u_3,y_1)(u_4,y_1)\}\} \cup \{(u_1,y_1)(u_3,y_1)(u_4,y_1),\\(u_1,y_1)(u_2,y_1)(u_4,y_1)\} \cup \{\mathcal{P}_{13} \setminus \{(u_2,y_2)(u_3,y_2)\}\} \cup \{(u_1,y_2)(u_3,y_2)(u_2,y_2)(u_4,y_2)\}$.

\vspace{1.5mm} Then $p(G\boxdot H_1)\leq \sum_{i=1}^{t_1} p(G \boxdot P_{y_{1_i},y_{2_i}}) = \sum_{i=1}^{t_1} p(G \Box P_2) \leq \sum_{i=1}^{t_1} m \leq \frac{mn_{o}^{1}(H)}{2}$.

\vspace{2mm}(2) For a $\mathcal{P}$-$r$-trail $W_1=y_1P_{y_1,y_2}\cdots y_{r}P_{y_{r}y_{r+1}}y_{r+1}$ of $\mathcal{P}(H)$, we will consider the three cases as follows:

\vspace{2mm}\noindent{\bf Case 1.} If $m\geq 6$, then $P_{m}\boxdot W_1$ can be decomposed into $m+3r-3$ paths as follows:

\vspace{2mm} Let $\mathcal{P}_{24}=\{(u_1,y_j)(u_3,y_j),(u_2,y_j)(u_4,y_j)\}$ for each $ j \in \{2,3, \cdots, r\}$.

$\mathcal{P}(G\boxdot W_1)=\mathcal{P}_{21} \cup \{\mathcal{P}_{22} \setminus \{(u_1,y_1)(u_2,y_1),(u_3,y_1)(u_4,y_1)\}\} \cup \{(u_1,y_1)(u_3,y_1)(u_4,y_1),\\ (u_1,y_1)(u_2,y_1)(u_4,y_1)\} \cup \{\mathcal{P}_{23} \setminus \{(u_2,y_{r+1})(u_3,y_{r+1})\}\} \cup  \{(u_1,y_{r+1})(u_3,y_{r+1})(u_2,y_{r+1})(u_4,y_{r+1})\}\\ \cup \mathcal{P}_{24}$.

\vspace{1.5mm} Then $p(G\boxdot H_2)\leq \sum_{i=1}^{t_2} p(G \boxdot W_{1_i}) \leq \sum_{i=1}^{t_2} (m+3r_i-3)\leq \sum_{i=1}^{t_2} \frac{m(r_i+1)}{2}\leq \frac{m(n_{o}^{2}(H)+n_{e}^{1}(H))}{2}$ where $r_i$ is the number of paths on $W_{1_i}$.

\vspace{2mm}\noindent{\bf Case 2.} If $m= 5$, then $P_{m}\boxdot W_1$ can be decomposed into at most $\frac{5r+3}{2}$ paths as follows:

\vspace{2mm} When $r$ is even, $\mathcal{P}_{24}=\{(u_4,y_1)(u_3,y_1)(u_1,y_1)P^1_{y_1,y_2}(u_1,y_2)(u_3,y_2)\} \cup \{\{(u_3,y_j)(u_1,y_j)\\ P^1_{y_j,y_{j+1}}(u_1,y_{j+1})(u_3,y_{j+1})\}$ for each $j\in \{3,5,\cdots, r-1\}$\}.

\vspace{2mm} When $r$ is odd, $\mathcal{P}_{24}=\{(u_4,y_1)(u_3,y_1)(u_1,y_1)P^1_{y_1,y_2}(u_1,y_2)(u_3,y_2)\} \cup \{(u_1,y_r)(u_3,y_r)\} \cup \{\{(u_3,y_j)(u_1,y_j)P^1_{y_j,y_{j+1}}(u_1,y_{j+1})(u_3,y_{j+1})\}$ for each $j\in \{3,5,\cdots, r-2\}$\}.

\vspace{2mm} $\mathcal{P}_{25}=\{(u_1,y_1)(u_2,y_1)(u_4,y_1)\} \cup \{\{(u_2,y_j)(u_4,y_j)\}$ for each $j\in \{2,3,\cdots, r\}$\}.

Let $\mathcal{P'}_{21}=\{P': P'=P-\{P\cap Q\}, P\in \mathcal{P}_{21}\ and\ Q\in \mathcal{P}_{24}\}$.

$\mathcal{P}(G\boxdot W_1)=\mathcal{P'}_{21} \cup \mathcal{P}_{24} \cup \mathcal{P}_{25} \cup \{(u_1,y_{r+1})(u_3,y_{r+1})(u_2,y_{r+1})(u_4,y_{r+1})(u_5,y_{r+1})\}$.

\vspace{2mm} Then $p(G\boxdot H_2)\leq \sum_{i=1}^{t_2} p(G \boxdot W_{1_i}) \leq \sum_{i=1}^{t_2} \frac{5r_i+3}{2}\leq \sum_{i=1}^{t_2} \frac{5(r_i+1)}{2} = \frac{5}{2}(n_{o}^{2}(H)+n_{e}^{1}(H))$ where $r_i$ is the number of paths on $W_{1_i}$.

\vspace{2mm}\noindent{\bf Case 3.} If $m= 4$, then $P_{m}\boxdot W_1$ can be decomposed into at most $2r+1$ paths as follows:

\vspace{2mm} When $r$ is odd, $\mathcal{P}_{24}=\{(u_4,y_1)(u_3,y_1)P^1_{y_1,y_2}(u_1,y_1)(u_1,y_2)(u_3,y_2)\} \cup \{(u_1,y_1)(u_2,y_1)\\ (u_4,y_1)P^4_{y_1,y_2}(u_4,y_2)(u_2,y_2)\} \cup \{(u_1,y_{r+1})(u_3,y_{r+1})(u_2,y_{r+1})(u_4,y_{r+1})\}$,

\vspace{1.5mm} $\mathcal{P}_{25}=\{(u_3,y_j)(u_1,y_j)P^1_{y_j,y_{j+1}}(u_1,y_{j+1})(u_3,y_{j+1})\} \cup \{(u_2,y_j)(u_4,y_j)P^4_{y_j,y_{j+1}}(u_4,y_{j+1})\\ (u_2,y_{j+1})\}$ for each $j\in \{3,4,\cdots, r-1\}$.

\vspace{2mm} When $r$ is even, $\mathcal{P}_{24}=\{(u_4,y_1)(u_3,y_1)(u_1,y_1)P^1_{y_1,y_2}(u_1,y_2)(u_3,y_2)\}\cup \{(u_1,y_1)(u_2,y_1)\\ (u_4,y_1)P^4_{y_1,y_2}(u_4,y_2)(u_2,y_2)\} \cup \{(u_3,y_r)(u_1,y_r)P^1_{y_r,y_{r+1}}(u_1,y_{r+1})(u_3,y_{r+1})(u_2,y_{r+1})(u_4,y_{r+1})\\ P^4_{y_r,y_{r+1}}(u_4,y_r)(u_2,y_r)\}$,

\vspace{1.5mm} $\mathcal{P}_{25}=\{(u_3,y_j)(u_1,y_j)P^1_{y_j,y_{j+1}}(u_1,y_{j+1})(u_3,y_{j+1})\} \cup \{(u_2,y_j)(u_4,y_j)P^4_{y_j,y_{j+1}}(u_4,y_{j+1})\\ (u_2,y_{j+1})\}$ for each $j\in \{3,4,\cdots, r-2\}$.

\vspace{2mm} Let $\mathcal{P'}_{21}=\{P': P'=P-\{P\cap Q\}, P\in \mathcal{P}_{21}\ and\ Q\in \{\mathcal{P}_{24}, \mathcal{P}_{25}\}\}$.

$\mathcal{P}(G\boxdot W_1)=\mathcal{P'}_{21} \cup \mathcal{P}_{24} \cup \mathcal{P}_{25}$.

\vspace{2mm} Then $p(G\boxdot H_2)\leq \sum_{i=1}^{t_2} p(G \boxdot W_{1_i})\leq \sum_{i=1}^{t_2}(2r_i+1)\leq \sum_{i=1}^{t_2} 2(r_i+1)= 2(n_{o}^{2}(H)+n_{e}^{1}(H))$ where $r_i$ is the number of paths on $W_{1_i}$.

\vspace{2mm}(3) For a $\mathcal{P}$-$r$-closed-trail $W_2=y_1P_{y_1,y_2}\cdots y_{r}P_{y_{r}y_{1}}y_{1}$ of $\mathcal{P}(H)$, we will consider the three cases as follows:

\vspace{2mm}\noindent{\bf Case 1.} If $m\geq 6$, then $P_{m}\boxdot W_2$ can be decomposed into $3r$ paths as follows:

\vspace{2mm}Let $\mathcal{P}_{32}=\{(u_1,y_j)(u_3,y_j),(u_2,y_j)(u_4,y_j)\}$ for each $ j \in \{1,3, \cdots, r\}$,

$\mathcal{P}(G\boxdot W_2)=\mathcal{P}_{31} \cup \mathcal{P}_{32}$.

\vspace{1.5mm} Then $p(G\boxdot H_3)\leq \sum_{i=1}^{t_3} p(G \boxdot W_{2_i}) \leq \sum_{i=1}^{t_3} 3r_i \leq \frac{mn_{e}^{2}(H)}{2}$ where $r_i$ is the number of paths on $W_{2_i}$.

\vspace{2mm}\noindent{\bf Case 2.} If $m= 5$, then $P_{m}\boxdot W_2$ can be decomposed into at most $\frac{5r}{2}$ paths as follows:

\vspace{2mm} When $r$ is even, $\mathcal{P}_{32}= \{(u_3,y_j)(u_1,y_j)P^1_{y_j,y_{j+1}}(u_1,y_{j+1})(u_3,y_{j+1})\}$ for each $j\in \{1,3,\cdots,\\ r-1\}$.

$\mathcal{P}_{33}=\{(u_2,y_j)(u_4,y_j)\}$ for each $j\in \{1,2,\cdots, r\}$.

Let $\mathcal{P'}_{31}=\{P': P'=P-\{P\cap Q\}, P\in \mathcal{P}_{31}\ and\ Q\in \mathcal{P}_{32}\}$.

$\mathcal{P}(G\boxdot W_2)=\mathcal{P'}_{31} \cup \mathcal{P}_{32} \cup \mathcal{P}_{33}.$

\vspace{2mm} When $r$ is odd, $\mathcal{P}_{32}=\{(u_3,y_j)(u_1,y_j)P^1_{y_j,y_{j+1}}(u_1,y_{j+1})(u_3,y_{j+1})\}$ for each $j\in \{2,4,\cdots,\\ r-1\}$,

$\mathcal{P}_{33}=\{(u_2,y_j)(u_4,y_j)\}$ for each $j\in \{1,2,\cdots, r\}$.

Let $\mathcal{P'}_{31}=\{\{\bigcup_{i=2}^{5} P_{y_{r-1},y_r}^i\}  \cup \{P_{y_r,y_1}^1\} \cup \{(u_1,y_r)(u_2,y_r),(u_2,y_{r-1})(u_3,y_{r-1}),(u_3,y_r)(u_4,y_r),\\ (u_4,y_{r-1})(u_5,y_{r-1})\} \cup \{(u_1,y_1)(u_3,y_1)\}\}\cup \{\{\bigcup_{i=2}^{5} P_{y_r,y_1}^i\}\cup \{(u_1,y_1)(u_2,y_1),(u_2,y_{r})(u_3,y_{r}),\\ (u_3,y_1)(u_4,y_2),(u_4,y_{r})(u_5,y_{r})\}\} \cup \{\{\bigcup_{i=1}^{5} P_{y_j,y_{j+1}}^i\} \cup \{(u_1,y_{j+1})(u_2,y_{j+1}),(u_2,y_j)(u_3,y_j),\\ (u_3,y_{j+1})(u_4,y_{j+1}), (u_4,y_j)(u_5,y_j)\}$ for each $j\in\{1, \cdots, r-2\}$\},

$\mathcal{P''}_{31}=\{P': P'=P-\{P\cap Q\}, P\in \mathcal{P'}_{31}\ and\ Q\in \mathcal{P}_{32}\}$.

$\mathcal{P}(G\boxdot W_2)=\mathcal{P''}_{31} \cup \mathcal{P}_{32} \cup \mathcal{P}_{33}.$

\vspace{1.5mm} Then $p(G\boxdot H_3)\leq \sum_{i=1}^{t_3} p(G \boxdot W_{2_i}) \leq \sum_{i=1}^{t_3} \frac{5r_i}{2} = \frac{5n_{e}^{2}(H)}{2}$ where $r_i$ is the number of paths on $W_{2_i}$.

\vspace{2mm}\noindent{\bf Case 3.} If $m= 4$, then $P_{m}\boxdot W_2$ can be decomposed into $2r$ paths as follows:

\vspace{2mm} When $r$ is even, $\mathcal{P}_{32}=\{(u_3,y_j)(u_1,y_j)P^1_{y_j,y_{j+1}}(u_1,y_{j+1})(u_3,y_{j+1})\} \cup \{(u_2,y_j)(u_4,y_j)\\ P^4_{y_j,y_{j+1}}(u_4,y_{j+1})(u_2,y_{j+1})\}$ for each $j\in \{1,3,\cdots, r-1\}$.

Let $\mathcal{P'}_{31}=\{P': P'=P-\{P\cap Q\}, P\in \mathcal{P}_{31}\ and\ Q\in \mathcal{P}_{32}\}$.

$\mathcal{P}(G\boxdot W_2)=\mathcal{P'}_{31} \cup \mathcal{P}_{32}$.

\vspace{2mm} When $r$ is odd, $\mathcal{P}_{32}=\{(u_2,y_j)(u_4,y_j)P^4_{y_j,y_{j+1}}(u_4,y_{j+1})(u_2,y_{j+1})\}\cup \{(u_3,y_j)(u_1,y_j)\\ P^1_{y_j,y_{j+1}}(u_1,y_{j+1})(u_3,y_{j+1})\}$ for each $j\in \{2,4,\cdots, r-1\}$.

Let $\mathcal{P'}_{31}=\{\{\bigcup_{i=2}^{3} P_{y_{r-1},y_r}^i\}\cup \{P_{y_r,y_1}^1\} \cup \{P_{y_r,y_1}^4\} \cup \{(u_1,y_r)(u_2,y_r),(u_2,y_{r-1})(u_3,y_{r-1}),\\ (u_3,y_r)(u_4,y_r)\}\cup \{(u_1,y_1)(u_3,y_1),(u_2,y_1)(u_4,y_1)\}\}\cup \{\{\bigcup_{i=2}^{3} P_{y_r,y_1}^i\}\cup \{(u_1,y_1)(u_2,y_1),\\ (u_2,y_{r})(u_3,y_{r}), (u_3,y_1)(u_4,y_2)\}\} \cup \{\{\bigcup_{i=1}^{4} P_{y_j,y_{j+1}}^i\} \cup \{(u_1,y_{j+1})(u_2,y_{j+1}),(u_2,y_j)(u_3,y_j),\\ (u_3,y_{j+1})(u_4,y_{j+1})\}$ for each $j\in\{1, \cdots, r-2\}$\},

$\mathcal{P''}_{31}=\{P': P'=P-\{P\cap Q\}, P\in \mathcal{P'}_{31}\ and\ Q\in \mathcal{P}_{32}\}$.

$\mathcal{P}(G\boxdot W_2)=\mathcal{P''}_{31} \cup \mathcal{P}_{32}$.

\vspace{2mm} Then $p(G\boxdot H_3)\leq \sum_{i=1}^{t_3} p(G \boxdot W_{2_i}) \leq \sum_{i=1}^{t_3} 2r_i \leq 2n_{e}^{2}(H)$ where $r_i$ is the number of paths on $W_{2_i}$.

\vspace{2mm} The above evidence shows that $G\Box H$ has a path decomposition $\mathcal{P}(G\Box H)$ with
\begin{eqnarray*}
p(G\Box H)&\leq&p(G\boxdot H_1)+p(G\boxdot H_2)+p(G\boxdot H_3)\\
&\leq&\frac{mn_{o}^{1}(H)}{2}+\frac{m(n_{o}^{2}(H)+n_e^{1}(H))}{2}+\frac{mn_e^{2}(H)}{2}\\
&=& \frac{mn}{2}.
\end{eqnarray*}
\end{proof}

\begin{lemma}\label{3.4}
Let $G$ be a connected graph of order $m$ and $H$ be a connected graph of order $n$. If $G \in \mathcal{G}$, then $p(G\Box H)\leq \frac{mn}{2}$.
\end{lemma}

\begin{proof}

Let $G=P_m+ \{u_1u_3, u_{m-2}u_m\}$ where $P_m=u_1u_2\cdots u_m$, $m\geq 6$.

\vspace{2mm}(1) For an odd-odd path $P_{y_1,y_2}$ of $\mathcal{P}(H)$, $G\boxdot P_{y_1,y_2}$ can be decomposed into $m$ paths as follows:

\vspace{2mm} If $m$ is even, then

$\mathcal{P'}_{12}=\{\mathcal{P}_{12} \setminus \{(u_1,y_1)(u_2,y_1), (u_{m-1},y_1) (u_m,y_1)\}\} \cup \{(u_2,y_1)(u_1,y_1)(u_3,y_1), (u_{m-2},y_1)\\ (u_m,y_1)(u_{m-1},y_1)\}$.

$\mathcal{P'}_{13}=\{\mathcal{P}_{13} \setminus \{(u_2,y_2)(u_3,y_2), (u_{m-2},y_2) (u_{m-1},y_2)\}\} \cup \{(u_1,y_2)(u_3,y_2)(u_2,y_2), (u_{m-1},y_2)\\ (u_{m-2},y_2)(u_m,y_2)\}$.

\vspace{2mm} If $m$ is odd, then

$\mathcal{P'}_{12}=\{\mathcal{P}_{12} \setminus \{(u_1,y_1)(u_2,y_1), (u_{m-2},y_1) (u_{m_1},y_1)\}\} \cup \{(u_2,y_1)(u_1,y_1)(u_3,y_1), (u_{m-1},y_1)\\ (u_{m-2},y_1)(u_m,y_1)\}$.

$\mathcal{P'}_{13}=\{\mathcal{P}_{13} \setminus \{(u_2,y_2)(u_3,y_2), (u_{m-1},y_2) (u_m,y_2)\}\} \cup \{(u_1,y_2)(u_3,y_2)(u_2,y_2), (u_{m-1},y_2)\\ (u_m,y_2)(u_{m-2},y_2)\}$

\vspace{2mm} $\mathcal{P}(G\boxdot P_{y_1,y_2})=\mathcal{P}_{11} \cup \mathcal{P'}_{12} \cup \mathcal{P'}_{13}$.

\vspace{2mm} Then $p(G\boxdot H_1)\leq \sum_{i=1}^{t_1} p(G \boxdot P_{y_{1_i},y_{2_i}})= \sum_{i=1}^{t_1} p(G \Box P_2)\leq \frac{mn_{o}^{1}(H)}{2}$.

\vspace{2mm}(2) For a $\mathcal{P}$-$r$-trail $W_1=y_1P_{y_1,y_2}\cdots y_{r}P_{y_{r}y_{r+1}}y_{r+1}$ of $\mathcal{P}(H)$, $G\boxdot W_1$ can be decomposed into $m+2r-2$ paths as follows:

\vspace{2mm} If $m$ is even, then

$\mathcal{P'}_{22}=\{\mathcal{P}_{22} \setminus \{(u_1,y_1)(u_2,y_1), (u_{m-1},y_1) (u_m,y_1)\}\} \cup \{(u_2,y_1)(u_1,y_1)(u_3,y_1), (u_{m-2},y_1)\\ (u_m,y_1)(u_{m-1},y_1)\}$.

$\mathcal{P'}_{23}=\{\mathcal{P}_{23} \setminus \{(u_2,y_{r+1})(u_3,y_{r+1}), (u_{m-2},y_{r+1}) (u_{m-1},y_{r+1})\}\} \cup \{(u_1,y_{r+1})(u_3,y_{r+1})\\ (u_2,y_{r+1}), (u_{m-1},y_{r+1}) (u_{m-2},y_{r+1})(u_m,y_{r+1})\}$.

\vspace{2mm} If $m$ is odd, then

$\mathcal{P'}_{22}=\{\mathcal{P}_{22} \setminus \{(u_1,y_1)(u_2,y_1), (u_{m-2},y_1) (u_{m_1},y_1)\}\} \cup \{(u_2,y_1)(u_1,y_1)(u_3,y_1), (u_{m-1},y_1)\\ (u_{m-2},y_1)(u_m,y_1)\}$.

$\mathcal{P'}_{23}=\{\mathcal{P}_{23} \setminus \{(u_2,y_{r+1})(u_3,y{r+1}), (u_{m-1},y_{r+1}) (u_m,y_{r+1})\}\} \cup \{(u_1,y_{r+1})(u_3,y_{r+1})\\ (u_2,y_{r+1}), (u_{m-1},y_{r+1}) (u_m,y_{r+1})(u_{m-2},y_{r+1})\}$

\vspace{2mm} $\mathcal{P}_{24}=\{(u_1,y_j)(u_3,y_j)(u_5,y_j)\}$ for each $ j \in \{2,3, \cdots, r\}$.

$\mathcal{P}(G\boxdot W_1)=\mathcal{P}_{21} \cup\mathcal{P'}_{22} \cup \{\mathcal{P'}_{13} \cup \mathcal{P}_{24}$.

\vspace{2mm} Then $p(G\boxdot H_2)\leq \sum_{i=1}^{t_2} p(G \boxdot W_{1_i}) \leq \sum_{i=1}^{t_2} (m+2r_i-2)\leq \sum_{i=1}^{t_2}\frac{m(r_i+1)}{2}\leq \frac{m(n_{o}^{2}(H)+n_{e}^{1}(H))}{2}$ where $r_i$ is the number of paths on $W_{1_i}$.

\vspace{2mm}(3) For a $\mathcal{P}$-$r$-closed-trail $W_2=y_1P_{y_1,y_2}\cdots y_{r}P_{y_{r}y_{1}}y_{1}$ of $\mathcal{P}(H)$, $G\boxdot W_2$ can be decomposed into $2r$ paths as follows:

\vspace{2mm} $\mathcal{P}_{32}=\{(u_1,y_j)(u_3,y_j), (u_{m-2},y_j)(u_m,y_j)\}$ for each $ j \in \{1,2, \cdots, r\}$.

$\mathcal{P}(G\boxdot W_2)=\mathcal{P}_{31} \cup \mathcal{P}_{32}$.

\vspace{2mm} Then $p(G\boxdot H_3)\leq \sum_{i=1}^{t_3} p(G \boxdot W_{2_i})\leq \sum_{i=1}^{t_3} 2r_i \leq \frac{mn_{e}^{2}(H)}{2}$ where $r_i$ is the number of paths on $W_{2_i}$.

\vspace{2mm} All in all, we conclude that $G\Box H$ has a path decomposition $\mathcal{P}(G\Box H)$ with

\begin{eqnarray*}
p(G\Box H)&\leq&p(G\boxdot H_1)+p(G\boxdot H_2)+p(G\boxdot H_3)\\
&\leq&\frac{mn_{o}^{1}(H)}{2}+\frac{m(n_{o}^{2}(H)+n_{e}^{1}(H))}{2}+\frac{mn_{e}^{2}(H)}{2}\\
&=& \frac{mn}{2}.
\end{eqnarray*}
\end{proof}

\section{\large Cartesian product of two general graph}

\vspace{3mm} Let $\mathcal{W}$ be the family of all graphs, which are obtained from a cycle and a path by identifying a vertex of the cycle with an end vertex of a separate path $P$.

\begin{lemma}\label{4.1}
Let $G$ be a connected graph of order $m$ and $H$ be a connected graph of order $n$. If $G \in \mathcal{W}$, then $p(G\Box H)\leq \frac{mn}{2}$.
\end{lemma}

\begin{proof}

Let $G=P_m + u_1u_l$ where $P_m=u_1u_2\cdots u_m$. If $m\geq 4$ and $m\geq l\geq 3$, by a similar argument as in above, we can fined a path decomposition of $G\Box H$.

\vspace{2mm}(1) For an odd-odd path $P_{y_1,y_2}$ of $\mathcal{P}(H)$, we will consider the following two cases.

\vspace{2mm} If $l$ is even, then $G\boxdot P_{y_1,y_2}$ can be decomposed into $m$ paths as follows:

$\mathcal{P}(G\boxdot P_{y_1,y_2})=\mathcal{P}_{11} \cup \{\mathcal{P}_{12} \setminus \{(u_1,y_1)(u_2,y_1),(u_{l-1},y_1)(u_l,y_1)\}\} \cup \{(u_2,y_1)(u_1,y_1)(u_l,y_1)\\ (u_{l-1},y_1)\} \cup \mathcal{P}_{13} \cup (u_1,y_2)(u_l,y_2)$.

\vspace{2mm} If $l$ is odd, then $G\boxdot P_{y_1,y_2}$ can be decomposed into $m$ paths as follows:

$\mathcal{P}(G\boxdot P_{y_1,y_2})=\mathcal{P}_{11} \cup \{\mathcal{P}_{12} \setminus (u_1,y_1)(u_2,y_1)\}
\cup \{(u_2,y_1)(u_1,y_1)(u_l,y_1)\} \cup \{\mathcal{P}_{13} \setminus (u_{l-1},y_2)(u_l,y_2)\} \cup \{(u_1,y_2)(u_l,y_2)(u_{l-1},y_2)\}$

\vspace{2mm} Then $p(G\boxdot H_1)\leq \sum_{i=1}^{t_1} p(G \boxdot P_{y_{1_i},y_{i_2}}) = \sum_{i=1}^{t_1} p(G \Box P_2) \leq \frac{mn_{o}^{1}(H)}{2}$.

\vspace{2mm}(2) For a $\mathcal{P}$-$r$-trail $W_1=y_1P_{y_1,y_2}\cdots y_{r}P_{y_{r}y_{r+1}}y_{r+1}$ of $\mathcal{P}(H)$, we will consider the following two cases.

\vspace{2mm} Let $\mathcal{P}_{24}=\{(u_1,y_j)(u_l,y_j)\}$ for each $ j \in \{2,3, \cdots, r\}$,

\vspace{2mm} If $l$ is even, then $G\boxdot W_1$ can be decomposed into $m+2r-2$ paths as follows:

$\mathcal{P}(G\boxdot W_1)=\mathcal{P}_{21} \cup \{\mathcal{P}_{22} \setminus \{(u_1,y_1)(u_2,y_1),(u_{l-1},y_1)(u_l,y_1)\}\} \cup \{(u_2,y_1)(u_1,y_1)(u_l,y_1)\\ (u_{l-1},y_1)\} \cup \mathcal{P}_{23} \cup (u_1,y_{r+1})(u_l,y_{r+1}) \cup \mathcal{P}_{24}$

\vspace{2mm} If $l$ is odd, then $G\boxdot W_1$ can be decomposed into $m+2r-2$ paths as follows:

$\mathcal{P}(G\boxdot W_1)=\mathcal{P}_{21} \cup \{\mathcal{P}_{22} \setminus (u_1,y_1)(u_2,y_1)\}
\cup \{(u_2,y_1)(u_1,y_1)(u_l,y_1)\} \cup \{\mathcal{P}_{23} \setminus (u_{l-1},y_{r+1})\\ (u_l,y_{r+1})\} \cup \{(u_1,y_{r+1})(u_l,y_{r+1})(u_{l-1},y_{r+1})\} \cup \mathcal{P}_{24}$

\vspace{2mm} Then $p(G\boxdot H_2)\leq \sum_{i=1}^{t_2} p(G \boxdot W_{1_i}) \leq \sum_{i=1}^{t_2}(m+2r_i-2)\leq \sum_{i=1}^{t_2}\frac{m(r_i+1)}{2}\leq \frac{m(n_{o}^{2}(H)+n_{e}^{1}(H))}{2}$ where $r_i$ is the number of paths on $W_{1_i}$.

\vspace{2mm}(3) For a $\mathcal{P}$-$r$-closed-trail $W_2=y_1P_{y_1,y_2}\cdots y_{r}P_{y_{r}y_{1}}y_{1}$ of $\mathcal{P}(H)$, $G\boxdot W_2$ can be decomposed into $2r$ paths as follows:

\vspace{2mm} Let $\mathcal{P}_{32}=\{(u_1,y_j)(u_l,y_j)\}$ for each $ j \in \{1,3, \cdots, r\}$,

\vspace{2mm} $\mathcal{P}(G\boxdot W_2)=\mathcal{P}_{31} \cup \mathcal{P}_{32}$.

\vspace{2mm} Then $p(G\boxdot H_3)\leq \sum_{i=1}^{t_3} p(G \boxdot W_{2_i}) \leq \sum_{i=1}^{t_3}2r_i\leq \sum_{i=1}^{t_3}\frac{m(r_i+1)}{2}\leq \frac{mn_{e}^{2}(H)}{2}$ where $r_i$ is the number of paths on $W_{2_i}$.

\vspace{2mm} For these reasons, we conclude that $G\Box H$ has a path decomposition $\mathcal{P}(G\Box H)$ with
\begin{eqnarray*}
p(G\Box H)&\leq&p(G\boxdot H_1)+p(G\boxdot H_2)+p(G\boxdot H_3)\\
&\leq&\frac{mn_{o}^{1}(H)}{2}+\frac{m(n_{o}^{2}(H)+n_e^{1}(H))}{2}+\frac{mn_e^{2}(H)}{2}\\
&=& \frac{mn}{2}.
\end{eqnarray*}

\vspace{2mm}Now we consider the case of $l=m=3$.

\vspace{2mm}(1) For an odd-odd path $P_{y_1,y_2}$ of $\mathcal{P}(H)$, $G\boxdot P_{y_1,y_2}$ can be decomposed into three paths as follows:

$(u_1,y_1)(u_2,y_1)(u_3,y_1)\cup P_{y_1,y_2}^3 \cup (u_3,y_2)(u_2,y_2)(u_1,y_2)$, $(u_3,y_1)(u_1,y_1)\cup P_{y_1,y_2}^1\cup (u_1,y_2)\\ (u_3,y_2)$ and  $P_{y_1,y_2}^2$.

Then $p(C_3\boxdot H_1)=\sum_{i=1}^{t_1} p(C_3 \boxdot P_{y_{1_i},y_{2_i}})=\frac{n_{o}^{1}(H)}{2} p(C_3 \Box P_2)\leq \frac{3n_{o}^{1}(H)}{2}$.

\vspace{2mm}(2) For a $\mathcal{P}$-$r$-trail $W_1=y_1P_{y_1,y_2}\cdots y_{r}P_{y_{r},y_{r+1}}y_{r+1}$ of $\mathcal{P}(H)$, $G\boxdot W_1$ can be decomposed into $r+2$ paths as follows:

\vspace{2mm} $P_{y_{i},y_{i+1}}^1 (u_1,y_{j+1}) (u_2,y_{j+1}) P_{y_{i+1},y_{i+2}}^2 (u_2,y_{j+2}) (u_3,y_{j+2}) P_{y_{i+2},y_{i+3}}^3 (u_3,y_{j+3})(u_1,y_{j+3})$ for each $ j \in \{1,2, \cdots, r\}$,

$(u_3,y_1)(u_1,y_1)(u_2,y_1)P_{y_{1},y_{2}}^2(u_2,y_2)(u_3,y_2)P_{y_{2},y_{3}}^3(u_3,y_3)(u_1,y_3)$,

and $(u_2,y_1)(u_3,y_1)P_{y_{1},y_{2}}^3(u_3,y_2)(u_1,y_2)$.

\vspace{2mm} Then $p(C_3\boxdot H_2)=\sum_{i=1}^{t_2} p(C_3 \boxdot W_{1_i}) \leq \sum_{i=1}^{t_2}(r_i+2) \leq \sum_{i=1}^{t_2}\frac{3(r_i+1)}{2} = \frac{3(n_{o}^{2}(H)+n_{e}^{1}(H))}{2}$.

\vspace{2mm}(3) For a $\mathcal{P}$-$r$-trail $W_2=y_1P_{y_1,y_2}\cdots y_{r}P_{y_{r}y_1}y_1$ of $\mathcal{P}(H)$, $G\boxdot W_2$ can be decomposed into $r$ paths as follows:

\vspace{2mm} $P_{y_{j},y_{j+1}}^1 (u_1,y_{j+1}) (u_2,y_{j+1}) P_{y_{j+1},y_{j+2}}^2 (u_2,y_{j+2}) (u_3,y_{j+2}) P_{y_{j+2},y_{j+3}}^3 (u_3,y_{j+3})(u_1,y_{j+3})$ for each $ j \in \{1,2, \cdots, r\}$, where the subscripts are taken modular $r$.

\vspace{2mm} Then $p(C_3\boxdot H_3)\leq \sum_{i=1}^{t_3} p(C_3 \boxdot W_{2_i})\leq \sum_{i=1}^{t_3} r_i\leq \frac{3n_{e}^{2}(H)}{2}$.

\vspace{2mm} Similarly, $G\Box H$ has a path decomposition $\mathcal{P}(G\Box H)$ with
\begin{eqnarray*}
p(G\Box H)&=& p(C_3\boxdot H_1)+p(C_3\boxdot H_2)+p(C_3\boxdot H_3)\\
&\leq& \frac{3n_{o}^{1}(H)}{2}+\frac{3(n_{o}^{2}(H)+n_e^{1}(H))}{2}+\frac{3n_{e}^{2}}{2}\\
&=& \frac{mn}{2}.
\end{eqnarray*}
\end{proof}

\begin{corollary}\label{4.2}
Let $m$ be an integer $m\geq 3$ and $H$ be a connected graph of order $n$. If $G=C_{m}\Box H$, then
$p(G)\leq \frac{mn}{2}$.
\end{corollary}

\vspace{1.5mm}Let $\mathcal{F}(G)$ be a family of edge-disjoint virtual-real path and virtual-real cycle of $G$ with
$E(G) = \{\cup_{P\in \mathcal{F}(G)}E(P)\}\cup  \{\cup_{C\in \mathcal{F}(G)}E(C)\}$. Let $v$ be any vertex in $G$. It is assigned to be real on the path in $\mathcal{F}(G)$ with $v$ as an end vertex if $v$ is odd, and is assigned to virtual for the remaining paths and cycles containing it; If $v$ is even, it is assigned to be real on unique a (randomly chosen) path or cycle in $\mathcal{F}(G)$ containing $v$, and is assigned to virtual for the remaining paths and cycles containing it. We say $\mathcal{F}(G)$ is a {\it virtual-real path cycle decomposition} of $G$. Specifically, if there are no virtual-real cycles in $\mathcal{F}(G)$, then $\mathcal{F}(G)$ is called a {\it virtual-real path decomposition}.

\begin{theorem}\label{4.3}
Let $G$ be a connected graph of order $m\geq 2$ and $H$ be a connected graph of order $n$. If $G$ has a virtual-real path cycle decomposition, then $p(G\Box H)\leq \frac{mn}{2}$.
\end{theorem}

\begin{proof}
We assume that graph $G$ can be decomposed into $t_1$ virtual-real paths and $t_2$ virtual-real cycles. Let $\mathcal{F}(G)=\{P_{l_i}^{s_i}: 1\leq i \leq t_1\}\cup \{C_{l_j}^{s_j}: 1\leq j \leq t_2\}$ where $P_{l_i}^{s_i}$ is a virtual-real path with $l_i$ vertices and $s_i$ real vertices and $C_{l_j}^{s_j}$ is a virtual-real cycle with $l_j$ vertices and $s_j$ real vertices. Based on the above facts, it follows that $\sum_{i=1}^{t_1}s_i+\sum_{j=1}^{t_2}s_i=m$. It is obvious that $E(G\Box H)=\{\cup_{i=1}^{t_1} E(P_{l_i}^{s_i}\boxdot H)\} \cup \{\cup_{j=1}^{t_2} E(C_{l_j}^{s_j}\boxdot H)\}$. By Lemma \ref{1}, $p(P_{l_i}^{s_i}\boxdot H)= p(P_{s_i}\Box H)$ and $p(C_{l_i}^{s_i}\boxdot H)= p(C_{s_i}\Box H)$. Moreover, by Theorem \ref{1.2} and Corollary \ref{4.2},

\begin{eqnarray*}
p(G\Box H)&\leq& \sum_{i=1}^{t_1} p(P_{l_i}^{s_i}\boxdot H)+\sum_{j=1}^{t_2} p(C_{l_j}^{s_j}\boxdot H)\\
&=& \sum_{i=1}^{t_1} p(P_{s_i}\Box H)+\sum_{j=1}^{t_2} p(C_{s_j}\Box H) \\
&\leq& \sum_{i=1}^{t_1}\frac{ns_i}{2}+\sum_{j=1}^{t_2}\frac{ns_j}{2}\\
&=& \frac{n}{2}(\sum_{i=1}^{t_1}s_i+\sum_{j=1}^{t_2}s_j)\\
&=& \frac{mn}{2}.
\end{eqnarray*}
\end{proof}

\section{\large Cartesian product of a graph and a unicyclic graph}

Let $C:=u_1u_2u_3 \cdots u_{l}u_1$ be a cycle. For any $1\leq i< j \leq l$, we denote by $C_{u_i,u_j}$ the path $u_iu_{i+1} \cdots u_j$ in $C$, and by $C_{u_j,u_i}$ the path $u_ju_{j+1} \cdots u_1\cdots u_i$ in $C$. The following lemma is what we are going to use.

\begin{theorem} [\cite{Chen2023}] \label{5.1}
For any tree $T$, $p(T)=\frac{n_o(T)}{2}$.
\end{theorem}

Now we are ready to prove one of our main results.
\begin{theorem}\label{5.2}
Let $G$ be a unicyclic graph of order $m\geq 3$ and $H$ be a connected graph of order $n$. Then $p(G\Box H)\leq \frac{mn}{2}$.
\end{theorem}
\begin{proof}

Let $C:= u_1u_2u_3 \cdots u_lu_1$, $3\leq l\leq m$ be the cycle of $G$. By Corollary 4.2, the result holds if $l=m$. So, let $m\leq l$. Let $G'=G-E(C)$. It is clearly that $G'$ is a forest. By Theorem \ref{5.1}, there exists a path decomposition $\mathcal{P}(G')$, such that each odd vertex is the end vertex of exactly one path of $\mathcal{P}(G')$ and each even vertex is not end vertex. Therefore, $G'$ has a virtual-real path decomposition $\mathcal{F}(G')$, and for any vertex $u_i\in V(G')$, there exists a virtual-real path $P^i$ of $\mathcal{F}(G')$ such that $u_i$ is real on $P^i$.

Let $S= \{u_i: u_i \in V(C)\ and\ d_G(u_i)> 2\}$ and $\mathcal{L}= \{P^i: u_i \in S\}$. Let $P'^i$ be the path obtained from $P^i$ by changing $u_i$ from real vertex to virtual vertex. Let $\mathcal{L'}= \{P'^i: P^i \in \mathcal{L}\}$. If all vertices on $C$ are even in $G$, then $G$ can be decomposed into $\mathcal{F'}(G')= \{\mathcal{F}(G')-\mathcal{L}\}\cup \mathcal{L'}$ and $C$. By Theorem \ref{4.3}, $p(G\Box H)\leq \frac{mn}{2}$.
Now let's assume that there is at least one odd vertex on $C$.

\vspace{3mm}\noindent{\bf Case 1.} At least two odd vertices in $C$.

\vspace{2mm} Suppose that $u_1$ and $u_2$ of $V(C)$ are odd vertices in $G$. Then $u_1$ and $u_2$ are the end vertices of $P^1$ and $P^2$, respectively. Moreover, we can fined a virtual-real path decomposition $\mathcal{F}(G)=\{\mathcal{F}(G')\setminus \{P^{1},P^{2}\}\}\cup \{P^{1}\cup C_{u_1,u_2},P^{2}\cup C_{u_2,u_1}\}$ of $G$ with $p(G)=p(G')=\frac{n_o(G)}{2}$. By Theorem \ref{1.3}, $p(G\Box H)\leq \frac{mn}{2}$.

\vspace{3mm}\noindent{\bf Case 2.} Only one odd vertex in $C$.

\vspace{2mm} Suppose that $u_1\in V(C)$ is an odd vertex in $G$. Then $G$ can be decomposed into $\mathcal{F'}(G')=\{\mathcal{F}(G') \setminus \{\mathcal{L} \cup P^1\}\} \cup \mathcal{L'}$ and $C'=C \cup P^1$. Suppose the number of real vertices of $C'$ is $s$. By Lemma \ref{4.1} and Theorem \ref{1.3}, we have

\begin{eqnarray*}
p(G\Box H)&\leq& \sum_{P\in \mathcal{F'}(G')} p(P\boxdot H)+p(C'\boxdot H)\\
&\leq& \frac{n(m-s)}{2}+\frac{ns}{2}\\
&=& \frac{mn}{2}.
\end{eqnarray*}
The proof is now finish.
\end{proof}

\section{\large Cartesian product of a graph and a bicyclic graph}

In this section, we discuss the path decomposition of the cartesian product of a graph and a bicyclic graph.
To simplify our statement, we give the following general definitions.

To {\it split} a vertex $v$ is to replace $v$ by two nonadjacent vertices, $v'$ and $v''$, and to replace each edge incident to $v$ by an edge incident to either $v'$ or $v''$ (but not both, unless it is a loop at
$v$), the other end of the edge remaining unchanged.

\begin{theorem}\label{6.1}
Let $G$ be a bicyclic graph of order $m$ and $H$ be a connected graph of order $n$. Then $p(G\Box H)\leq \frac{mn}{2}$.
\end{theorem}

\begin{proof}
Suppose $C_{l_1}:= u_1u_2u_3 \cdots u_{l_1}u_1$ and $C_{l_2}:= v_1v_2v_3 \cdots v_{l_2}v_1$ are the cycles of graph $G$. Let $U=V(C_{l_1})\cup V(C_{l_2})$ and $D=G[U]$. If $G'=G-E(D)$, then $G'$ is a forest. By Theorem \ref{5.1}, there exists a path decomposition $\mathcal{P}(G')$, such that each odd vertex is the end vertex of exactly one path of $\mathcal{P}(G')$ and each even vertex is not end vertex. Therefore, $G'$ has a virtual-real path decomposition $\mathcal{F}(G')$, and for any vertex $u_i\in V(C_{l_1})$, there exists a virtual-real path $P^i$ of $\mathcal{F}(G')$ such that $u_i$ is real on $P^i$, for any vertex $v_j\in V(C_{l_2})$, there exists a virtual-real path $Q^j$ of $\mathcal{F}(G')$ such that $v_j$ is real on $Q^j$. Let $P'^i$ be the path obtained from $P^i$ by changing $x_i$ from real vertex to virtual vertex where $x_i\in U$. Now we consider the path decomposition of edge of $D$.

\vspace{2mm} (1) $|N(C_{l_1}) \cap N(C_{l_2})|=0$.

\vspace{2mm} Let $U'=U-\{u_1,v_1\}$. Let $S= \{x_i: x_i \in U'\ and\ d_{G'}(x_i)\ is\ even\ and\ at \ least\ two\}$, $\mathcal{L}= \{P^i: x_i \in S\}$ and $\mathcal{L'}= \{P'^i: P^i \in \mathcal{L}\}$. Assuming $P_{u_1,v_1}$ is the path connecting $u_1$ and $v_1$ in $G$. Now we consider the following subcases.

\vspace{2mm}\noindent{\bf Case 1.} There exists an odd vertex in $C_{l_1}-u_1$.

\vspace{2mm} Without loss of generality, suppose $u_2$ is the odd vertex in $C_{l_1}-u_1$. We consider the following subcases.

\vspace{2mm}\noindent{\bf Subcase 1.1.} There exists an odd vertex in $C_{l_2}-v_1$, say $v_2$.

\vspace{2mm} If both $d_G(u_1)$ and $d_G(v_1)$ are even, then must exists a path $P_{x,y}$ in $\mathcal{F}(G')$ that contains $P_{u_1,v_1}$, where $u_1$ and $v_1$ are virtual. Then we can obtain a virtual-real path decomposition $\mathcal{F}(G)=\{\mathcal{F}(G')\setminus \{P^2, Q^2, P_{x,y}\}\} \cup \{P^2\cup C_{u_2,u_1} \cup P_{u_1,v_1} \cup C_{v_1,v_2} \cup Q^2\} \cup \{P_{x,u_1} \cup C_{u_1,u_2}\} \cup \{P_{y,v_1} \cup C_{v_2,v_1}\}$. By Theorem \ref{1.3}, $p(G\Box H)\leq \frac{mn}{2}$.

\vspace{2mm} If both $d_G(u_1)$ and $d_G(v_1)$ are odd, then $P_{u_1,v_1}$ ia a path in $\mathcal{F}(G')$. Thus, we can obtain a virtual-real path decomposition $\mathcal{F}(G)=\{\mathcal{F}(G')\setminus \{P^2, Q^2, P_{u_1,v_1}\}\} \cup \{P^2\cup C_{u_1,u_2}\} \cup \{C_{u_2,u_1} \cup P_{u_1,v_1} \cup C_{v_2,v_1}\} \cup \{C_{v_1,v_2} \cup Q^2\}$. By Theorem \ref{1.3}, $p(G\Box H)\leq \frac{mn}{2}$.

\vspace{2mm} Now we consider the case of $d_G(u_1)$ and $d_G(v_1)$ have different parity. Suppose that $d_G(u_1)$ is even and $d_G(v_1)$ is odd, then exists a path $P_{x,v_1}$ in $\mathcal{F}(G')$ that contains $P_{u_1,v_1}$, where $u_1$ is virtual and $v_2$ is real. Then we can obtain two virtual-real paths $P_{x,v_1}\cup C_{v_1,v_2}$ and $Q^2\cup C_{v_2v_1}$. Now we consider the path decomposition of $C_{l_1}$.

\vspace{2mm} If there exists an odd vertex in $C_{l_1}-\{u_1,u_2\}$, say $u_3$, then $P^2 \cup C_{l_1} \cup P^3$ can be decomposed into virtual-real paths $P^2 \cup C_{u_2,u_3}$ and $P^3\cup C_{u_3,u_2}$. By Theorem \ref{1.3}, $p(G\Box H)\leq \frac{mn}{2}$. If there is no odd vertex in $C_{l_1}-u_1$, then $P^2 \cup C_{l_1}$ is an element of $\mathcal{W}$. Thus, $G$ can be decomposed into $\mathcal{F'}(G')=\{\mathcal{F}(G') \setminus \{\mathcal{L}\cup\{P^2, Q^2, P_{x,v_1}\}\}\}\cup \mathcal{L'} \cup \{P_{x,v_1}\cup C_{v_1,v_2}\} \cup \{Q^2\cup C_{v_2v_1}\}$ and $C'_{l_1}=P^2 \cup C_{l_1}$. Suppose the number of real vertices of $C'_{l_1}$ is $s_1$. By Lemma \ref{4.1} and Theorem \ref{1.3}, we have

\begin{eqnarray*}
p(G\Box H)&\leq& \sum_{P\in \mathcal{F'}(G')} p(P\boxdot H)+p(C'_{l_1}\boxdot H)\\
&\leq& \frac{n(m-s_1)}{2}+\frac{ns_1}{2}\\
&=& \frac{mn}{2}.
\end{eqnarray*}

\vspace{2mm} Next let's assume that no odd vertex in $C_{l_2}-v_1$ and consider the following two cases.

\vspace{2mm}\noindent{\bf Subcase 1.2.} There exists an odd vertex in $V(C_{l_1})-\{u_1,u_2\}$, say $u_3$.

\vspace{2mm} If both $d_G(u_1)$ and $d_G(v_1)$ are even, then must exists a path $P_{x,y}$ in $\mathcal{F}(G')$ that contains $P_{u_1,v_1}$, where $u_1$ and $v_1$ are virtual. Thus, we can obtain a virtual-real path cycle decomposition $\mathcal{F}(G)=\{\mathcal{F}(G')\setminus \{\mathcal{L}\cup \{P^2, P^3\}\}\} \cup \mathcal{L} \cup \{P^{2}\cup C_{u_2,u_3}\} \cup \{P^3\cup C_{u_3,u_2}\} \cup C_{l_2}$. By Theorem \ref{4.3}, $p(G\Box H)\leq \frac{mn}{2}$.

\vspace{2mm} If both $d_G(u_1)$ and $d_G(v_1)$ are odd, then $P_{u_1,v_1}$ is a path in $\mathcal{F}(G')$. Then $G$ can be decomposed into $\mathcal{F'}(G')= \{\mathcal{F}(G') \setminus \{\mathcal{L}\cup \{P^2, P^3\}\}\} \cup \mathcal{L'} \cup \{P^2\cup C_{u_2, u_3}\} \cup \{P^3\cup C_{u_3, u_2}\}$ and $C'_{l_2}=P_{u_1,v_1} \cup C_{l_2}$. It is obvious that $C'_{l_2} \in \mathcal{W}$. Suppose the number of real vertices of $C'_{l_2}$ is $s_2$. By Lemma \ref{4.1} and Theorem \ref{1.3}, we have

\begin{eqnarray*}
p(G\Box H)&\leq& \sum_{P\in \mathcal{F'}(G')} p(P\boxdot H)+p(C'_{l_2}\boxdot H)\\
&\leq& \frac{n(m-s_2)}{2}+\frac{ns_2}{2}\\
&=& \frac{mn}{2}.
\end{eqnarray*}

\vspace{2mm} If $d_G(u_1)$ is even and $d_G(v_1)$ is odd, then exists a path $P_{x,v_1}$ in $\mathcal{F}(G')$ that contains $P_{u_1,v_1}$, where $u_1$ is virtual and $v_2$ is real. Hence, $G$ can be decomposed into $\mathcal{F'}(G')=\{\mathcal{F}(G') \setminus \{\mathcal{L}\cup \{P^2, P^3, P_{x,v_1}\}\}\} \cup \mathcal{L'} \cup \{P^2\cup C_{u_2,u_3}\} \cup \{P^3\cup C_{u_3,u_2}\}$ and $C'_{l_2}=P_{x,v_1} \cup C_{l_2}$. It is obvious that $C'_{l_2} \in \mathcal{W}$. Similar to the above, we have $p(G\Box H)\leq \frac{mn}{2}$.

\vspace{2mm} If $d_G(u_1)$ is odd and $d_G(v_1)$ is even, then exists a path $P_{u_1,y}$ in $\mathcal{F}(G')$ that contains $P_{u_1,v_1}$, where $u_1$ is real and $v_2$ is virtual. Thus, we can obtain a virtual-real path cycle decomposition $\mathcal{F}(G)=\{\mathcal{F}(G')\setminus \{\mathcal{L}\cup \{P^2, P^3\}\}\} \cup \mathcal{L'} \cup \{P^2\cup C_{u_2,u_3}\} \cup \{P^3\cup C_{u_3,u_2}\} \cup C_{l_2}$. By Theorem \ref{4.3}, $p(G\Box H)\leq \frac{mn}{2}$.

\vspace{2mm}\noindent{\bf Subcase 1.3.} There exists no odd vertex in $V(C_{l_1})-\{u_1,u_2\}$.

\vspace{2mm} If both $d_G(u_1)$ and $d_G(v_1)$ are even, then exists a path $P_{x,y}$ in $\mathcal{F}(G')$ that  contains $P_{u_1,v_1}$, where $u_1$ and $v_1$ are virtual. Thus, we can obtain a virtual-real path cycle decomposition $\mathcal{F}(G)=\{\mathcal{F}(G')\setminus \{\mathcal{L}\cup \{P^2, P_{x,y}\}\}\} \cup \mathcal{L'} \cup \{P^2\cup C_{u_2,u_1} \cup P_{u_1,y}\} \cup \{C_{u_1,u_2}\cup P_{x,u_1}\} \cup C_{l_2}$.  By Theorem \ref{4.3}, $p(G\Box H)\leq \frac{mn}{2}$.

\vspace{2mm} If both $d_G(u_1)$ and $d_G(v_1)$ are odd, then $P_{u_1,v_1}$ is a path in $\mathcal{F}(G')$. Then $G$ can be decomposed into $\mathcal{F'}(G')= \{\mathcal{F}(G') \setminus \{\mathcal{L}\cup P^2\}\} \cup \mathcal{L'} \cup \{P^2\cup C_{u_1,u_2}\}$ and $C'_{l_2}=C_{u_2,u_1} \cup P_{u_1,v_1} \cup C_{l_2}$. It is obvious that $C'_{l_2} \in \mathcal{W}$. Suppose the number of real vertices of $C'_{l_2}$ is $s_2$. By Lemma \ref{4.1} and Theorem \ref{1.3}, we have

\begin{eqnarray*}
p(G\Box H)&\leq& \sum_{P\in \mathcal{F'}(G')} p(P\boxdot H)+p(C'_{l_2}\boxdot H)\\
&\leq& \frac{n(m-s_2)}{2}+\frac{ns_2}{2}\\
&=& \frac{mn}{2}.
\end{eqnarray*}

\vspace{2mm} If $d_G(u_1)$ is even and $d_G(v_1)$ is odd, then exists a path $P_{x,v_1}$ in $\mathcal{F}(G')$ that contains $P_{u_1,v_1}$, where $u_1$ is virtual and $v_2$ is real. Then $G$ can be decomposed into $\mathcal{F'}(G')=\{\mathcal{F}(G') \setminus \{\mathcal{L}\cup \{P^2, P_{x,v_1}\}\}\} \cup \mathcal{L'}$ and $C'_{l_1}=P^2 \cup C_{l_1}$ and $C'_{l_2}=P_{x,v_1} \cup C_{l_2}$. It is obvious that $C'_{l_2}, C'_{l_2} \in \mathcal{W}$. Similar to the above, we have $p(G\Box H)\leq \frac{mn}{2}$.

\vspace{2mm} If $d_G(u_1)$ is odd and $d_G(v_1)$ is even, then exists a path $P_{u_1,y}$ in $\mathcal{F}(G')$ that contains $P_{u_1,v_1}$, where $u_1$ is real and $v_2$ is virtual. Thus, we can obtain a virtual-real path cycle decomposition $\mathcal{F}(G)=\{\mathcal{F}(G')\setminus \{\mathcal{L}\cup \{P^2, P_{u_1,y}\}\}\} \cup \mathcal{L'} \cup \{P^2\cup C_{u_2,u_1} \cup P_{u_1,y}\} \cup C_{u_1,u_2} \cup C_{l_2}$. By Theorem \ref{4.3}, $p(G\Box H)\leq \frac{mn}{2}$.

\vspace{2mm}\noindent{\bf Case 2.} There exists no odd vertex in $C_{l_1}-u_1$.

\vspace{2mm} If there are two odd vertices in $C_{l_2}-v_1$, then the proof is similar to Subcase 1.2. If there is an odd vertex in $C_{l_2}-v_1$, then the proof is similar to Subcase 1.3. We now suppose there is no odd vertex in $C_{l_2}-v_1$ and consider the following subcases.

\vspace{2mm} If both $d_G(u_1)$ and $d_G(v_1)$ are even, then exists a path $P_{x,y}$ in $\mathcal{F}(G')$ that  contains $P_{u_1,v_1}$, where $u_1$ and $v_1$ are virtual. Then $G$ can be decomposed into $\mathcal{F'}(G')=\mathcal{F}(G') \setminus \{\mathcal{L}\} \cup \mathcal{L'}$, $C_{l_1}$ and $C_{l_2}$. By Theorem \ref{4.3}, $p(G\Box H)\leq \frac{mn}{2}$.

\vspace{2mm} If $d_G(u_1)$ and $d_G(v_1)$ have different parity, suppose that $d_G(u_1)$ is even and $d_G(v_1)$ is odd, then exists a path $P_{x,v_1}$ in $\mathcal{F}(G')$ that contains $P_{u_1,v_1}$, where $u_1$ is virtual and $v_2$ is real. Hence, $G$ can be decomposed into $\mathcal{F'}(G')=\{\mathcal{F}(G') \setminus \{\mathcal{L} \cup P_{x,v_1}\}\} \cup \mathcal{L'} \cup C_{l_1}$ and $C'_{l_2}=P_{x,v_1} \cup C_{l_2}$. It is obvious that $C'_{l_2} \in \mathcal{W}$. Suppose the number of real vertices of $C'_{l_2}$ is $s_2$. By Lemma \ref{4.1} and Theorem \ref{4.3}, we have

\begin{eqnarray*}
p(G\Box H)&\leq& \sum_{P\in \mathcal{F'}(G')} p(P\boxdot H)+p(C'_{l_2}\boxdot H)\\
&\leq& \frac{n(m-s_2)}{2}+\frac{ns_2}{2}\\
&=& \frac{mn}{2}.
\end{eqnarray*}

\vspace{2mm} If both $d_G(u_1)$ and $d_G(v_1)$ are odd, then $P_{u_1,v_1}$ is a path in $\mathcal{F}(G')$. If $max\{l_1,l_2\}\geq 4$, suppose $l_1\geq 4$, then $C_1$ is a virtual-real cycle where $u_1$ is virtual. Then $G$ can be decomposed into $\mathcal{F'}(G')=\{\mathcal{F}(G') \setminus \{\mathcal{L}\cup P_{x,v_1}\}\} \cup \mathcal{L'} \cup C_{l_1}$ and $C'_{l_2}=P_{x,v_1} \cup C_{l_2}$. Similar to the above, we have $p(G\Box H)\leq\frac{mn}{2}$.

Now we consider the case of $l_1=l_2=3$. Let $\mathcal{F'}(G')=\{\mathcal{F}(G') \setminus \{\mathcal{L} \cup P_{u_1,v_1}\}\} \cup \mathcal{L'}$ and $D'=C_{l_1}\cup P_{u_1,v_1} \cup C_{l_2}$. Clearly, $D'\in \mathcal{G}$. Suppose the number of real vertices of $D'$ is $s$. By Lemma \ref{3.4} and Theorem \ref{1.3}, we have

\begin{eqnarray*}
p(G\Box H)&\leq& \sum_{P\in \mathcal{F'}(G')} p(P\boxdot H)+p(D'\boxdot H)\\
&\leq& \frac{n(m-s)}{2}+\frac{ns}{2}\\
&=& \frac{mn}{2}.
\end{eqnarray*}

\vspace{2mm} (1) $|N(C_{l_1}) \cap N(C_{l_2})|=1$

\vspace{2mm} Suppose $u_1=v_1$ and $U'=U-u_1$. Let $S= \{x_i: x_i \in U'\ and\ d_{G'}(x_i)\ is\ even\ and\ at \ least\\ two\}$, $\mathcal{L}= \{P^i: x_i \in S\}$ and $\mathcal{L'}= \{P'^i: P^i \in \mathcal{L}\}$. Now we consider the following cases.

\vspace{2mm}\noindent{\bf Case 1.} There exists an odd vertex in $V(C_{l_1})-u_1$.

\vspace{2mm} Without loss of generality, suppose $u_2$ is the odd vertex in $C_{l_1}-u_1$. We consider the following subcases.

\vspace{2mm}\noindent{\bf Subcase 1.1.} There exists an odd vertex in $C_{l_2}-v_1$, say $v_2$.

\vspace{2mm} We can obtain a virtual-real path decomposition $\mathcal{F}(G)=\{\mathcal{F}(G')\setminus \{P^2, Q^2\}\} \cup \{P^2\cup C_{u_2u_1}\cup C_{u_1v_2}\} \cup \{Q^2\cup C_{v_2v_1}\cup C_{u_1u_2}\}$ of $G$. By Theorem \ref{1.3}, $p(G\Box H)\leq \frac{mn}{2}$.

\vspace{2mm} Next let's assume that no odd vertex in $C_{l_2}-v_1$ and consider the following two cases.

\vspace{2mm}\noindent{\bf Subcase 1.2.} There exists an odd vertex in $C_{l_1}-\{u_1,u_2\}$, say $u_3$.

\vspace{2mm} Let $G''=G- E(C_{l_2})$. Then we can obtain a virtual-real path decomposition $\mathcal{F}(G'')=\{\mathcal{F}(G')\setminus \{P^{2}, P^3\}\} \cup \{P^{2}\cup C_{u_2,u_3}\} \cup \{P^3\cup C_{u_3,u_2}\}$ of $G''$. Next we focus on the decomposition of $C_{l_2}$.

\vspace{2mm} If $d_{G'}(v_1)$ is even, then there are at least three even vertices on $C_{l_2}$.
Hence, we can obtain a virtual-real path cycle decomposition $\mathcal{F}(G)=\{\mathcal{F}(G'') \setminus \{\mathcal{L}\cup P^1\}\} \cup \{\mathcal{L'}\cup P'^1\} \cup C_{l_2}$ of $G$. By Theorem \ref{4.3}, $p(G\Box H)\leq \frac{mn}{2}$.

\vspace{2mm} If $d_{G'}(u_1)$ is odd, then $u_1$ is the end vertex of $P^1$. If $l_2\geq 4$, similar to above, then $C_{l_2}$ is a virtual-real cycle with at least three real vertices. Hence, $G\Box H$ has we desirous path decomposition. If $l_2= 3$, then $G$ can be decomposed into $\mathcal{F'}(G'')=\mathcal{F}(G'') \setminus \{\mathcal{L}\cup P^1\} \cup \mathcal{L'}$ and $C'_{l_2}=C_{l_2}\cup P^1$. It is clear that $C'_{l_2}\in \mathcal{W}$. Suppose the number of real vertex of $C'_{l_2}$ is $s_2$. By Lemma \ref{4.1} and Theorem \ref{4.3}, we have

\begin{eqnarray*}
p(G\Box H)&\leq& \sum_{P\in \mathcal{F'}(G'')} p(P\boxdot H)+p(C'_{l_2}\boxdot H)\\
&\leq& \frac{n(m-s_2)}{2}+\frac{ns_2}{2}\\
&=& \frac{mn}{2}.
\end{eqnarray*}

\vspace{2mm}\noindent{\bf Subcase 1.3.} There exists no odd vertex in $C_{l_1}-\{u_1,u_2\}$.

\vspace{2mm} If $l_2\geq 4$, similar to above, then $C_{l_2}$ is a cycle with at least three real vertices. Let $G''= G' + E(C_{l_2})$. Now we consider the decomposition of $C_{l_1}$.

\vspace{2mm} If $d_G(u_1)$ is odd, then we can obtain a virtual-real path cycle decomposition $\mathcal{F}(G)=\{\mathcal{F}(G') \setminus \{\mathcal{L} \cup \{P^1,P^2\}\}\} \cup \mathcal{L'} \cup \{P^1\cup C_{u_1,u_2}\} \cup \{P^2\cup C_{u_2,u_1}\} \cup C_{l_2}$ of $G$. By Theorem \ref{4.3}, $p(G\Box H)\leq \frac{mn}{2}$. If $d_G(u_1)$ is even, then we can obtain $C'_{l_1}=C_{l_1}\cup P^2$ which belong in $\mathcal{W}$. Then $G$ can be decomposed into $\mathcal{F'}(G'')=\{\mathcal{F}(G')\setminus \{\mathcal{L}\cup P^2\}\} \cup \mathcal{L'} \cup C_{l_2}$ and $C'_{l_1}$. Suppose the number of real vertices of $C'_{l_1}$ is $s_1$. By Lemma \ref{4.1} and Theorem \ref{4.3}, we have

\begin{eqnarray*}
p(G\Box H)&\leq& \sum_{P\in \mathcal{F'}(G'')} p(P\boxdot H)+p(C'_{l_1}\boxdot H)\\
&\leq& \frac{n(m-s_1)}{2}+\frac{ns_1}{2}\\
&=& \frac{mn}{2}.
\end{eqnarray*}

Now we consider the case of $l_2= 3$. If $l_1\geq 4$, then we can obtain $C'_{l_1}=C_{l_1}\cup P^2$ where $u_1$ is virtual. If $d_G(u_1)$ is even, then $C_{l_2}$ is a cycle with three real vertices. If $d_G(u_1)$ is odd, then we can obtain $C'_{l_2}=C_{l_2}\cup P^1$. It is clear that $C'_{l_1}, C'_{l_2}\in \mathcal{W}$. Similar to the above, $p(G\Box H)\leq \frac{mn}{2}$.

\vspace{2mm} If $l_1=3$ and $d_G(u_1)$ is odd, then $G$ can be decomposed into $\mathcal{F'}(G')=\{\mathcal{F}(G')\setminus \{\mathcal{L} \cup \{P^2, P^1\}\}\} \cup \mathcal{L'} \cup \{P^2 \cup C_{u_2,u_1} \cup P^1\}$ and $C'_{l_2}= C_{u_1,u_2} \cup C_{l_2}$. It is clear that $C'_{l_2} \in \mathcal{W}$. Similar to the above, $p(G\Box H)\leq \frac{mn}{2}$. If $l_1=3$ and $d_G(u_1)$ is even, then $D\cup P^2$ is a member of $\mathcal{A}$. Thus, $G$ can be decomposed into $\mathcal{F'}(G')=\{\mathcal{F}(G')\setminus \{\mathcal{L} \cup \{P^1\cup P^2\}\}\} \cup \{\mathcal{L'} \cup P'^1\}$ and $D'=D\cup P^2$. Suppose the number of real vertices of $D'$ is $s$. By Lemma \ref{3.1} and Theorem \ref{4.3}, we have

\begin{eqnarray*}
p(G\Box H)&\leq& \sum_{P\in \mathcal{F'}(G')} p(P\boxdot H)+p(D' \boxdot H)\\
&\leq& \frac{n(m-s)}{2}+\frac{ns}{2}\\
&=& \frac{mn}{2}.
\end{eqnarray*}

\vspace{2mm}\noindent{\bf Case 2.} There exists no odd vertex in $C_{l_1}-u_1$.

\vspace{2mm} If there are two odd vertices in $C_{l_2}-\{v_1\}$, then the proof is similar to Subcase 1.2. If there is an odd vertex in $C_{l_2}-\{v_1\}$, then the proof is similar to Subcase 1.3. We now suppose there is no odd vertex in $C_{l_2}-\{v_1\}$ and consider the following subcases.

\vspace{2mm}\noindent{\bf Subcase 2.1.} There exist $u_i, v_j\in U'$ such that $d_{G'}(u_i)\geq 2$ and $d_{G'}(v_j)\geq 2$.

\vspace{2mm} It is fact that $u_i$ and $v_j$ are internal vertex of $P^i$ and $Q^j$ of $\mathcal{P}(G)$, respectively. Suppose that $P^i=P_{x,u_i}\cup P_{u_i,y}$, where $x$ and $y$ are the end vertices of $P^i$ and $P_{x,u_i}$ and $P_{y,u_i}$ are the subpaths of $P^i$. Similarly, $Q^j=Q_{z,v_j}\cup Q_{w,v_j}$. Now we obtain a virtual-real path decomposition $\mathcal{F}(G)=\{\mathcal{F}(G') \setminus \{P^{i}, Q^{j}\}\} \cup \{P_{x,u_i}\cup C_{u_i,u_1}\cup C_{u_1,v_j}\cup Q_{z,v_j}\} \cup \{P_{y,u_i}\cup C_{u_1,u_i}\cup C_{v_j,u_1}\cup Q_{w,v_j}\}$ of $G$. By Theorem \ref{1.3}, $p(G\Box H)\leq \frac{mn}{2}$.

\vspace{2mm}\noindent{\bf Subcase 2.2.} There exist no $u_i, v_j\in U'$ such that $d_{G'}(u_i)\geq 2$ and $d_{G'}(v_j)\geq 2$.

\vspace{2mm} First, we consider the case of $max\{l_1,l_2\}\geq 4$. Without loss of generality, we suppose that $l_1\geq 4$. It is clear that $C_{l_1}$ is a virtual-real cycle with $u_1$ being virtual. If $u_1$ is even in $G$, then $C_{l_2}$ also a virtual-real cycle with $u_1$ being real. By Theorem \ref{4.3}, $p(G\Box H)\leq \frac{mn}{2}$. Otherwise, $C_{l_2}\cup P^1$ is a member of $\mathcal{W}$. By Lemma \ref{4.1} and Theorem \ref{4.3}, $p(G\Box H)\leq \frac{mn}{2}$.

\vspace{2mm} Now we consider the case of $l_1=l_2= 3$. If $u_1$ is even in $G$, then $D\in \mathcal{A}$. Let $\mathcal{F'}(G')=\{\mathcal{F}(G')\setminus \{\mathcal{L}\cup P^1\}\} \cup \{\mathcal{L'}\cup P'^1\}$. By Lemma \ref{3.1} and Theorem \ref{4.3}, We have

\begin{eqnarray*}
p(G\Box H)&\leq& \sum_{P\in \mathcal{F'}(G')} p(P\boxdot H)+p(D \boxdot H)\\
&\leq& \frac{n(m-5)}{2}+\frac{5n}{2}\\
&=& \frac{mn}{2}.
\end{eqnarray*}

If $u_1$ is odd in $G$, then $u_1$ is virtual on $D$. To split $u_1$ to $u_{11}$ and $u_{12}$, then both of them are virtual. We can obtain a virtual-real cycle $C'=u_2u_3u_{11}v_3v_2u_{12}$, where $u_{11}$ and $u_{12}$ are virtual. By Lemma \ref{1} and Corollary \ref{4.2}, we can obtain a path decomposition of $C'\boxdot H$ with $p(C'\boxdot H)=p(C_4\Box H)$. By Lemma \ref{4.1} and Theorem \ref{1.3}, $p(G\Box H)\leq \frac{mn}{2}$.

\vspace{2mm}(2) $|N(C_{l_1}) \cap N(C_{l_2})|\geq 2$.

%It is fact that there are three paths from $u_1$ to $u_t$ in $G$.
\vspace{2mm} Let $Q_1:= u_1u_2\cdots u_t=v_1v_2\cdots v_t$, where $2 \leq t \leq min\{l_1, l_2\}$ be the shortest path from $u_1$ to $u_t$ in $G$. Then $u_1, \cdots, u_t$ are the common vertices of $C_{l_1}$ and $C_{l_2}$. Let $U'=U-\{u_1u_2\cdots u_t\}$. Let $S= \{x_i: x_i \in U'\ and\ d_{G'}(x_i)\ is\ even\ and\ at \ least\ two\}$, $\mathcal{L}= \{P^i: x_i \in S\}$ and $\mathcal{L'}= \{P'^i: P^i \in \mathcal{L}\}$. We consider the following subcases.

\vspace{3mm}\noindent{\bf Subcase 2.1.} All vertices of $U'$ are even in $G$.

\vspace{2mm} Let $C=u_{t}u_{t+1}\cdots u_{l_1}u_1v_{l_2}v_{l_2-1} \cdots v_{t+1}$. If both $u_1$ and $u_t$ are odd in $G'$, then we can obtain a new virtual-real path $Q'_1=P^1\cup Q_1\cup P^t$, where $u_1$ and $u_t$ are virtual. Then we can obtain a virtual-real path cycle decomposition $\mathcal{F}(G)=\{\mathcal{F}(G') \setminus \{\mathcal{L} \cup \{P^1, P^t\}\}\}\cup \mathcal{L'}\cup Q'_1\cup C$ of $G$. By Theorem \ref{4.3}, $p(G\Box H)\leq \frac{mn}{2}$.

\vspace{2mm} If only one of $u_1$ and $u_t$ is odd in $G'$, suppose $u_1$ is odd, then we can obtain a new virtual-real path $Q'_1=P^1\cup Q_1$, where $u_1$ virtual and $u_t$ is real. Then $\mathcal{F}(G)= \{\mathcal{F}(G') \setminus \{\mathcal{L} \cup \{P^1,P^t\}\}\}\cup \mathcal{L'}\cup Q'_1\cup P'^t\cup C$. By Theorem \ref{4.3}, $p(G\Box H)\leq \frac{mn}{2}$.

\vspace{2mm} If both $u_1$ and $u_t$ are even in $G'$, then we will consider the following two cases. If $|V(D)-V(Q_1)|\geq 3$, then $Q_1$ is a new path where $u_1$ and $u_t$ are real, and $C$ is a cycle with at least three real vertices. Then $\mathcal{F}(G)=\{\mathcal{F}(G') \setminus \{\mathcal{L} \cup \{P^1, P^t\}\}\cup \{\mathcal{L'}\cup \{P'^1, P'^t\}\}\cup Q_1\cup C$. By Theorem \ref{4.3}, $p(G\Box H)\leq \frac{mn}{2}$. If $|V(D)-V(Q_1)|=2$, then $D\in \{\mathcal{B}, \mathcal{F}\}$. Let $\mathcal{F'}(G')=\{\mathcal{F}(G') \setminus \{\mathcal{L} \cup \{P^1, P^t\}\}\cup \{\mathcal{L'}\cup \{P'^1, P'^t\}\}$. By lemma \ref{3.2}, lemma \ref{3.3} and Theorem \ref{4.3}, we have

\begin{eqnarray*}
p(G\Box H)&\leq& \sum_{P\in \mathcal{F'}(G')} p(P\boxdot H)+p(D\boxdot H)\\
&\leq& \frac{n(m-|D|)}{2}+\frac{n|D|}{2}\\
&=& \frac{mn}{2}.
\end{eqnarray*}

\vspace{3mm}\noindent{\bf Subcase 2.2.} At least two odd vertices in $U'$.

\vspace{2mm} Without loss of generality, we suppose $u_{l_1}$ and $v_{l_2}$ are odd in $G$. If both $u_1$ and $u_t$ are odd in $G'$, then we can obtain a new virtual-real path $Q'_1=P^1\cup Q_1\cup P^t$ where $u_1$ and $u_t$ are virtual. Then we can fined a virtual-real path decomposition $\mathcal{F}(G)=\{\mathcal{F}(G') \setminus \{P^{l_1}, Q^{l_2}, P^1, P^t\}\} \cup \{P^{l_1}\cup C_{u_{l_1},u_1}\cup C_{v_{l_2},u_1}\} \cup \{Q^{l_2}\cup C_{v_t,v_{l_2}}\cup C_{u_t,u_{l_1}}\} \cup Q'_1$ of $G$. By Theorem \ref{1.3}, $p(G\Box H)\leq \frac{mn}{2}$.

\vspace{2mm} If only one of $u_1$ and $u_t$ is odd in $G'$, we suppose $u_1$ is odd, then we can obtain a new virtual-real path $Q'_1=P^1\cup Q_1$ where $u_1$ virtual is $u_t$ is real. Then  $\mathcal{F}(G)=\{\mathcal{F}(G') \setminus \{P^{l_1}, Q^{l_2}, P^1\}\} \cup \{P^{l_1}\cup C_{u_{l_1},u_1}\cup C_{v_{l_2},u_1}\} \cup \{Q^2\cup C_{v_t,v_{l_2}}\cup C_{u_t,u_{l_1}}\} \cup Q'_1$. By Theorem \ref{1.3}, $p(G\Box H)\leq \frac{mn}{2}$.

\vspace{2mm} If both $u_1$ and $u_t$ are even in $G'$, then $Q_1$ is a new path. Then $\mathcal{F}(G)=\{\mathcal{F}(G') \setminus \{P^{l_1}, Q^{l_2}\}\} \cup \{P^{l_1}\cup C_{u_{l_1},u_1}\cup C_{v_{l_2},u_1}\} \cup \{Q^{l_2}\cup C_{v_t,v_{l_2}}\cup C_{u_t,u_{l_1}}\} \cup Q_1$. By Theorem \ref{1.3}, $p(G\Box H)\leq \frac{mn}{2}$.

\vspace{3mm}\noindent{\bf Subcase 2.3.} Only one odd vertex in $U'$.

\vspace{2mm} Without loss of generality, we suppose $u_{l_1}$ is odd in $G'$. If both $u_1$ and $u_t$ are odd in $G'$, then we can fined a virtual-real path decomposition $\mathcal{F}(G)=\{\mathcal{F}(G') \setminus \{P^1, P^t, P^{l_1}\}\} \cup \{P^1\cup C_{u_1,u_{l_1}}\} \cup \{P^{l_1}\cup C_{u_{l_1},u_1}\cup C_{v_t,v_1}\cup P^t\}$. By Theorem \ref{1.3}, $p(G\Box H)\leq \frac{mn}{2}$.

\vspace{1.5mm} If only one of $u_1$ and $u_t$ is odd in $G'$, suppose $u_1$ is odd in $G$, then we can fined a virtual-real path decomposition $\mathcal{F}(G)=\{\mathcal{F}(G') \setminus \{P^1, P^{l_1}\}\} \cup \{P^1\cup C_{u_1,u_{l_1}}\} \cup \{P^{l_1}\cup C_{u_{l_1},u_1}\cup C_{v_t,v_1}\}$. By Theorem \ref{1.3}, $p(G\Box H)\leq \frac{mn}{2}$.

\vspace{1.5mm} If both $u_1$ and $u_t$ are even in $G'$, then $Q_1$ is a new path. Let $G''=G'+Q_1$ and $D'=\{D\setminus Q_1\}\cup P^{l_1}$. If $|V(D)-V(Q_1)|\geq 3$, then $D'\in \{\mathcal{W}\}$. Let $\mathcal{F}(G'')=\{\mathcal{F}(G') \setminus \{\mathcal{L} \cup P^{l_1}\}\} \cup \{\mathcal{L'}\} \cup Q_1$. Then the graph $G$ can been decompose into $\mathcal{F}(G'')$ and $D'$. Suppose the number of real vertices of $D'$ is $s$, by Lemma \ref{4.1} and Theorem \ref{1.3}, we have

\begin{eqnarray*}
p(G\Box H)&\leq& \sum_{P\in \mathcal{F}(G'')} p(P\boxdot H)+p(D'\boxdot H)\\
&\leq& \frac{n(m-s)}{2}+\frac{ns}{2}\\
&=& \frac{mn}{2}.
\end{eqnarray*}

If $|V(D)-V(Q_1)|= 2$, then $D\cup P^{l_1}\in \{\mathcal{B}, \mathcal{F}\}$. Let $D''=D\cup P^{l_1}$. Then $G$ can be decomposed into $\mathcal{F'}(G')=\{\mathcal{F}(G')\setminus \{\mathcal{L}\cup P^{l_1}\}\} \cup \mathcal{L'}$ and $D''$. Suppose the number of real vertices of $D''$ is $s$, by Lemma \ref{3.2}, Lemma \ref{3.3} and Theorem \ref{1.3}, we have

\begin{eqnarray*}
p(G\Box H)&\leq& \sum_{P\in \mathcal{F'}(G')} p(P\boxdot H)+p(D''\boxdot H)\\
&\leq& \frac{n(m-s)}{2}+\frac{ns}{2}\\
&=& \frac{mn}{2}.
\end{eqnarray*}
\end{proof}

\noindent {\bf Data availability statement}

\vspace{2mm}\noindent Data sharing not applicable to this article as no data sets were generated or analysed during the current study.

\vspace{3mm}\noindent {\bf Declarations of competing interest}

The authors declare that they have no known competing financial interests or personal relationships that could have appeared to influence the work reported in this paper.

\end{document}